\documentclass{article}
\usepackage{amsfonts}
\usepackage{amssymb}
\usepackage{amsmath}
\usepackage{graphicx}

\newif\ifpics
\picstrue 

\usepackage{color} 
   \definecolor{cites}{rgb}{0.50 , 0.00 , 0.00}  
   \definecolor{urls} {rgb}{0.00 , 0.00 , 0.50}  
   \definecolor{links}{rgb}{0.00 , 0.00 , 0.50}   

\usepackage{hyperref} 
\hypersetup{  
      colorlinks=true,   
      citecolor=cites,   
      urlcolor=urls,     
      linkcolor=links,   
      pdfstartview=FitH,       
      bookmarksopen=false      
}

\newcommand\C{{\mathbb C}}
\newcommand\D{{\mathbb D}}
\newcommand\E{{\mathcal E}}
\newcommand\cN{{\mathcal N}}

\newcommand\cNfinp{{\mathcal N}_{{\rm fin},p}}
\newcommand\cNfinq{{\mathcal N}_{{\rm fin},q}}
\newcommand\Mfinp{{\mathcal M}_{{\rm fin},p}}
\newcommand\Mfinq{{\mathcal M}_{{\rm fin},q}}
\newcommand\Sfin{{\Sigma}_{{\rm fin},\eps}}
\newcommand\I{{\mathbb I}}

\newcommand\M{{\mathcal M}}

\newcommand\N{{\mathbb N}}

\renewcommand\S{{\Sigma}}
\newcommand\T{{\mathbb T}}
\newcommand\R{{\mathbb R}}

\newcommand\Z{{\mathbb Z}}

\newcommand\sign{{\rm sign}}
\newcommand\eps{{\varepsilon}}
\newcommand\spec{{\rm spec}\,}  
\newcommand\specn{{\rm spec}}   
\newcommand\speps{{\rm spec}_\eps}
\newcommand\spess{{\rm spec}_{\rm ess}\,}

\newcommand\opsp{{\sigma^{\sf op}}}

\newcommand\dist{{\rm dist}}

\newcommand\diag{{\rm diag}}

\newcommand\conv{{\rm conv\,}}
\newcommand\convn{{\rm conv}}
\newcommand\ri{{\rm i}}

\newcommand\ind{{\rm ind\,}}
\newcommand\wind{{\rm wind}}

\newcommand\Ein{E_{\rm in}}
\newcommand\Eout{E_{\rm out}}
\newcommand\cond{{\rm cond}}

\newcommand\ovD{{\overline{\mathbb D}}}

\newcommand\PE{{\rm \Psi E}} 
\newcommand\Mfin{M_{\rm fin}}

\newcommand\BDO{{\rm BDO}}
\newcommand\Eins{E_{\cap}}
\newcommand\Eall{E_{\cup}}

\newcommand\toH{
   \unitlength0.1ex
   \begin{picture}(30,15)
   \put(13,16){\makebox(0,0)[]{\tiny\rm H}}
   \put(15,5){\makebox(0,0)[]{$\to$}}
   \end{picture}
}

\newtheorem{theorem}{Theorem}[section]
\newtheorem{lemma}[theorem]{Lemma}
\newtheorem{corollary}[theorem]{Corollary}
\newtheorem{proposition}[theorem]{Proposition}

\newenvironment{remark}
 {\par\noindent\refstepcounter{theorem}{\bf Remark \thetheorem}}
 {\raisebox{1mm}{\framebox{}}\pagebreak[2]}

\newenvironment{example}
 {\par\noindent\refstepcounter{theorem}{\bf Example \thetheorem}}
 {\raisebox{1mm}{\framebox{}}\pagebreak[2]}

\newcommand\Proofend{\rule{2mm}{2mm}}

\newenvironment{proof}
 {\par\noindent{\bf Proof.}}
 {\Proofend\pagebreak[2]}

\newenvironment{proofof}[1]
 {\par\noindent{\bf Proof of #1.}}
 {\Proofend\pagebreak[2]}

\numberwithin{figure}{section}  

\begin{document}
\title{\bf Coburn's Lemma and the Finite Section Method for Random Jacobi Operators}
\author{{\sc Simon N. Chandler-Wilde}\footnote{Email: {\tt S.N.Chandler-Wilde@reading.ac.uk}}\quad  and\quad {\sc Marko Lindner}\footnote{Email: {\tt lindner@tuhh.de}}}
\date{\today}
\maketitle
\begin{quote}
\renewcommand{\baselinestretch}{1.0}
\footnotesize {\sc Abstract.}
We study the spectra and pseudospectra of semi-infinite and bi-infinite tridiagonal random matrices and their finite principal submatrices, in the case where each of the three diagonals varies over a separate compact set, say $U,V,W\subset\C$. Such matrices are sometimes termed stochastic Toeplitz matrices $A_+$ in the semi-infinite case and stochastic Laurent matrices $A$ in the bi-infinite case. Their spectra, $\S=\spec A$ and $\S_+=\spec A_+$, are independent of $A$ and $A_+$ as long as $A$ and $A_+$ are pseudoergodic (in the sense of E.B.~Davies, {\em Commun.~Math.~Phys.} {\bf 216} (2001), 687--704), which holds almost surely in the random case. This was shown in Davies (2001) for $A$;  that the same holds for $A_+$ is one main result of this paper. Although the computation of $\S$ and $\S_+$ in terms of $U$, $V$ and $W$ is intrinsically difficult, we give upper and lower spectral bounds, and we explicitly  compute a set $G$ that fills the gap between $\S$ and $\S_+$ in the sense that $\S\cup G=\S_+$.
We also show that the invertibility of one (and hence all) operators $A_+$ implies the invertibility -- and uniform boundedness of the inverses -- of all finite tridiagonal square matrices with diagonals varying over $U$, $V$ and $W$. This implies that the so-called finite section method for the approximate solution of a system $A_+x=b$ is applicable as soon as $A_+$ is invertible, and that the finite section method for estimating the spectrum of $A_+$ does not suffer from spectral pollution.
Both results illustrate that tridiagonal stochastic Toeplitz operators share important properties of (classical) Toeplitz operators. Indeed, one of our main tools is a new stochastic version of the Coburn lemma for classical Toeplitz operators, saying that a stochastic tridiagonal Toeplitz operator, if Fredholm, is always injective or surjective.
In the final part of the paper we bound and compare the norms, and the norms of inverses, of bi-infinite, semi-infinite and finite tridiagonal matrices over $U$, $V$ and $W$. This, in particular, allows the study of the resolvent norms, and hence the pseudospectra, of these operators and matrices.
\end{quote}


\section{Introduction and Main Results} \label{sec:intro}
In this paper we study so-called {\sl Jacobi operators} over three sets $U$, $V$ and $W$, meaning bi- and semi-infinite matrices of the form
\begin{equation} \label{eq:A}
A = \left(\begin{array}{ccccccc} \ddots&\ddots\\
\ddots&v_{-2}&w_{-2}\\
&u_{-1}&v_{-1}&w_{-1}\\
\cline{4-4}
&&u_{0}&\multicolumn{1}{|c|}{v_0}&w_0\\
\cline{4-4}
&&&u_{1}&v_1&w_1\\
&&&&u_2&v_2&\smash{\ddots}\\
&&&&&\ddots&\ddots
\end{array}\right)\quad \textrm{and}\quad
A_+ = \left(\begin{array}{ccccc}
v_{1}&w_{1}\\
u_{2}&v_{2}&w_{2}\\
&u_{3}&v_3&w_3\\
&&u_{4}&v_4&\smash{\ddots}\\
&&&\ddots&\ddots
\end{array}\right)
\end{equation}
with entries $u_i\in U$, $v_i\in V$ and $w_i\in W$ for all $i$ under consideration. The sets $U$, $V$ and $W$ are nonempty and compact subsets of the complex plane $\C$, and the box marks the matrix entry of $A$ at $(0,0)$.  We will be especially interested in the case where the matrix entries are random (say i.i.d.) samples from $U$, $V$ and $W$. Trefethen et al.~\cite{TrefContEmb} call the operator $A$ a {\sl stochastic Laurent matrix} in this case and  $A_+$ a {\sl stochastic Toeplitz matrix}. We will adopt this terminology which seems appropriate given that one of our aims is to highlight parallels between the analysis of standard and stochastic Laurent and Toeplitz matrices.

It is known that the spectrum of $A$ depends only on the sets $U$, $V$, and $W$, as long as $A$ is pseudoergodic in the sense of Davies \cite{Davies2001:PseudoErg}, which holds almost surely if $A$ is stochastic (see the discussion below). Via a version, which applies to stochastic Toeplitz matrices, of the famous Coburn lemma \cite{Coburn} for (standard) Toeplitz matrices, a main result of this paper is to show that, with the same assumption of pseudoergodicity implied by stochasticity, also the spectrum of $A_+$ depends only on $U$, $V$, and $W$. Moreoever, we tease out very explicitly what the difference is between the spectrum of a stochastic Laurent matrix and the spectrum of the corresponding stochastic Toeplitz matrix. (The difference will be that certain `holes' in the spectrum of the stochastic Laurent case may be `filled in' in the stochastic Toeplitz case, rather similar to the standard Laurent and Toeplitz cases.)

A second main result is to show that infinite linear systems, in which the matrix, taking one of the forms \eqref{eq:A}, is a stochastic Laurent or Toeplitz matrix, can be solved effectively by the standard finite section method, provided only that the respective infinite matrices are invertible. In particular, our results show that, if the stochastic Toeplitz matrix is invertible, then every finite $n\times n$ matrix formed by taking the first $n$ rows and columns of $A_+$ is invertible, and moreover the inverses are uniformly bounded. Again, this result, which can be interpreted as showing that the finite section method for stochastic Toeplitz matrices does not suffer from spectral pollution (cf.~\cite{MarlettaNaboko2014}), is reminiscent of the standard Toeplitz case.

{\bf Related work.} The study of random Jacobi operators and their spectra has one of its main roots in the famous Anderson model \cite{Anderson58,Anderson61} from the late 1950's. In the 1990's the study of a non-selfadjoint (NSA) Anderson model, the Hatano-Nelson model \cite{HatanoNelson1997,NelsonShnerb1998}, led to a series of papers on NSA random operators and their spectra, see e.g. \cite{GoldKoru,Davies2001:PseudoErg,Davies2001:SpecNSA,Martinez2007}.
Other examples of NSA models are discussed in \cite{CicutaContediniMolinari2000,CicutaContediniMolinari2002,TrefContEmb,LiBiDiag}: one example that has attracted significant recent attention (and which, arguably, has a particularly intriguing spectrum) is the randomly hopping particle model due to Feinberg and Zee \cite{FeinZee97,FeinZee99,CWChonchaiyaLindner2011,CWDavies2011,CCL2,Hagger:NumRange,Hagger:dense,Hagger:symmetries,CWHa2015}. A comprehensive discussion of this history, its main contributors, and many more references can be found in Sections 36 and 37 of \cite{TrefEmbBook}.

 A theme of many of these studies \cite{TrefContEmb,TrefEmbBook,CWChonchaiyaLindner2011,CWDavies2011,CCL2,Hagger:dense,CWHa2015}, a theme that is central to this paper, is the relationships between the spectra, norms of inverses, and pseudospectra of random operators, and the corresponding properties of the random matrices that are their finite sections. Strongly related to this (see the discussion in the `Main Results' paragraphs below) is work on the relation between norms of inverses and pseudospectra of finite and infinite {\em classical} Toeplitz
and Laurent matrices \cite{ReichelTref,Boe94,BoeGruSilb97}. In between the classical and stochastic Toeplitz cases, the same issues have also been studied for randomly perturbed Toeplitz and
Laurent operators \cite{BoeEmSok02,BoeEmSok03a,BoeEmSok03b,BoeEmLi}.  This paper, while focussed on the specific features of the random case, draws strongly on results on the finite section method in much more general contexts: see \cite{LiBook} and the `Finite sections' discussion below. With no assumption of randomness, the finite section method for a particular class of NSA perturbations of (selfadjoint) Jacobi matrices is analysed recently in \cite{MarlettaNaboko2014}.

We will build particularly on two recent studies of random Jacobi matrices $A$ and $A_+$ and their finite sections. In \cite{Hagger:NumRange} it is shown that the closure of the numerical range of these operators is the convex hull of the spectrum, this holding whether or not $A$ and $A_+$ are normal operators. Further, an explicit expression for this numerical range is given: see \eqref{eq:hagger} below.  In \cite{LindnerRoch2010} progress is made in bounding the spectrum and understanding the finite section method applied to solving infinite linear systems where the matrix is a tridiagonal stochastic Laurent or Toeplitz matrix.   This last paper is the main starting point for this present work and we recall key notations and concepts that we will build on from \cite{LindnerRoch2010} in the following paragraphs.

{\bf Matrix notations.} We understand $A$ and $A_+$ as linear operators, again denoted by $A$ and $A_+$, acting boundedly, by matrix-vector multiplication, on the standard spaces $\ell^p(\Z)$ and $\ell^p(\N)$ of bi- and singly-infinite complex sequences with $p\in [1,\infty]$. The sets of all operators $A$ and $A_+$ from \eqref{eq:A} with entries $u_i\in U$, $v_i\in V$ and $w_i\in W$ for all indices $i$ that occur will be denoted by $M(U,V,W)$ and $M_+(U,V,W)$, respectively.  For $n\in \N$ the set of $n\times n$ tridiagonal matrices with subdiagonal entries in $U$, main diagonal entries in $V$ and superdiagonal entries in $W$ (and all other entries zero) will be denoted $M_n(U,V,W)$, and we set $\Mfin(U,V,W):= \cup_{n\in \N} M_n(U,V,W)$. For $X=\ell^{p}(\I)$ with $\I=\N$ or $\Z$, we call a bounded linear operator $A:X\to X$ a {\sl band-dominated operator} and write $A\in\BDO(X)$ if $A$ is the limit, in the operator norm on $X$, of a sequence of band operators (i.e. bounded operators on $X$ that are induced by infinite matrices with finitely many nonzero diagonals). For $A=(a_{ij})_{i,j\in \I}\in \BDO(X)$, $A^\top=(a_{ji})_{i,j\in \I}\in \BDO(X)$ denotes the transpose of $A$. The boundedness of the sets $U$, $V$ and $W$ implies that every operator in $M(U,V,W)$ or $M_+(U,V,W)$ is bounded and hence band-dominated (of course even banded). We use $\|\cdot\|$ as the notation for the norm of an element of $X=\ell^p(\I)$, whether $\I=\N$, $\Z$, or $-\N=\{\dots, -2,-1\}$, or $\I$ is a finite set, and use the same notation for the induced operator norm of a bounded linear operator on $X$ (which is the induced norm of a finite square matrix if $\I$ is finite). If we have a particular $p\in[1,\infty]$ in mind, or want to emphasise the dependence on $p$, we will write $\|\cdot \|_p$ instead of $\|\cdot\|$.

{\bf Random alias pseudoergodic operators.} Our particular interest is random operators in $M(U,V,W)$ and $M_+(U,V,W)$. We model randomness by the following deterministic concept: we call $A\in M(U,V,W)$ {\sl pseudoergodic} and write $A\in\PE(U,V,W)$ if every finite Jacobi matrix over $U$, $V$ and $W$ can be found, up to arbitrary precision, as a submatrix of $A$. Precisely, for every $B\in\Mfin(U,V,W)$ and every $\eps>0$ there is a finite square submatrix $C$ of $A$ such that $\|B-C\|<\eps$. Here $C$ is composed of rows $i=k+1,...,k+n$ and columns $j=k+1,...,k+n$ of $A$, where $k$ is some integer and $n$ is the size of $B$. By literally the same definition we define semi-infinite pseudoergodic matrices $A_+$ and denote the set of these matrices by $\PE_+(U,V,W)$.
Pseudoergodicity was introduced by Davies \cite{Davies2001:PseudoErg} to study spectral properties of random operators while eliminating probabilistic arguments. Indeed, if all matrix entries in \eqref{eq:A} are chosen independently (or at least not fully correlated) using probability measures whose supports are $U$, $V$ and $W$, then, with probability one, $A$ and $A_+$ in \eqref{eq:A} are pseudoergodic.

{\bf Fredholm operators, spectra, and pseudospectra.} Recall that a bounded linear operator $B:X\to Y$ between Banach spaces is a {\sl Fredholm operator} if the dimension, $\alpha(B)$, of its null-space is finite and the codimension, $\beta(B)$, of its image
in $Y$ is finite. In this case, the image of $B$ is closed in $Y$ and the integer $\ind B:=\alpha(B)-\beta(B)$ is called the {\sl index} of $B$. Equivalently, $B$ is a Fredholm operator if it has a so-called regularizer $C:Y\to X$ modulo compact operators, meaning that $BC-I_Y$ and $CB-I_X$ are both compact.
For a bounded linear operator $B$ on $\ell^p(\I)$ with $\I\in\{\Z,\N,-\N\}$, we write $\specn^p B$ and $\specn_{\rm ess}^p\,B$ for the sets of all $\lambda\in\C$ for which $B-\lambda I$ is, respectively, not invertible or not a Fredholm operator on $\ell^p(\I)$. Because $A$ and $A_+$ in \eqref{eq:A} are band matrices, their spectrum and essential spectrum do not depend on the underlying $\ell^p$-space \cite{Kurbatov,Li:Wiener,Roch:ellp}, so that we will just write $\spec A$ and $\spess A$ for operators $A\in M(U,V,W)$ -- and similarly for $A_+\in M_+(U,V,W)$. Even more, if $A_+\in M_+(U,V,W)$ or $A\in M(U,V,W)$ is Fredholm as an operator on $\ell^p(\N)$ or $\ell^p(\Z)$, respectively, both its Fredholm property and index are independent of $p$ \cite{Roch:ellp,Li:Wiener}.

Following notably \cite{TrefEmbBook}, we will be interested in not just the spectrum and essential spectrum, but also the {\em $\eps$-pseudospectrum}, the union of the spectrum with those $\lambda\in \C$ where the resolvent is well-defined with norm $>\eps^{-1}$. Precisely, for an $n\times n$ complex matrix $B$, or a bounded linear operator $B$ on $\ell^p(\I)$ with $\I\in\{\Z,\N\}$, we define, for $1\le p\le \infty$ and $\eps>0$,
\begin{equation} \label{eq:PSdef}
\speps^p B\ :=\ \left\{\lambda\in\C\ :\ \|(B-\lambda I)^{-1}\|_p>\eps^{-1}\right\},
\end{equation}
with the convention that $\|A^{-1}\|_p:=\infty$ if $A$ is not invertible (so that $\specn^p B\subset\speps^p B$). While $\specn^p B$ is independent of $p$ (and so abbreviated as $\spec B$), the set $\speps^p B$ depends on $p$ in general. It is a standard result (see \cite{TrefEmbBook} for this and the other standard results we quote) that
\begin{equation} \label{eq:specenc}
\spec B + \eps \D \subset \speps^p B,
\end{equation}
with $\D:=\{z\in\C:|z|<1\}$ the open unit disk.
If $p=2$ and $B$ is a normal matrix or operator then equality holds in \eqref{eq:specenc}. Clearly, for $0<\eps_1<\eps_2$, $\spec B \subset \specn_{\eps_1}^p B \subset \specn_{\eps_2}^p B$, and $\spec B = \bigcap_{\eps>0} \specn_\eps^p B$ (for $1\le p\le \infty$).  Where $\overline{S}$ denotes the closure of $S\subset \C$, a deeper result, see the discussion in \cite{Shargorodsky08} (summarised in \cite{CCL2}), is that
\begin{equation} \label{eq:Spec}
\overline{\speps^p B} = \mathrm{Spec}_\eps^p B := \left\{\lambda\in\C\ :\ \|(B-\lambda I)^{-1}\|_p\ge\eps^{-1}\right\}.
\end{equation}
Interest in pseudospectra has many motivations \cite{TrefEmbBook}. One is that $\speps^p B$ is the union of $\specn(B+T)$ over all perturbations $T$ with  $\|T\|_p<\eps$. Another is that, unlike $\spec B$ in general, the pseudospectrum depends continuously on $B$ with respect to the standard Hausdorff metric (see \eqref{eq:Haus} below).

{\bf Limit operators.} A main tool of our paper, and of \cite{LindnerRoch2010}, is the notion of limit operators. For $A=(a_{ij})_{i,j\in\Z}\in\BDO(X)$ with $X=\ell^p(\Z)$ and $h_1,h_2,...$ in $\Z$ with $|h_n|\to\infty$ we say that $B=(b_{ij})_{i,j\in\Z}$ is a {\sl limit operator} of $A$ if, for all $i,j\in\Z$,
\begin{equation}\label{eq:limop}
a_{i+h_n,j+h_n}\to b_{ij}\quad\textrm{as}\quad n\to\infty.
\end{equation}
The boundedness of the diagonals of $A$ ensures (by Bolzano-Weierstrass) the existence of such sequences $(h_n)$ and the corresponding limit operators $B$. From $A\in \BDO(X)$ it follows that $B\in \BDO(X)$. The closedness of $U$, $V$ and $W$ implies that $B\in M(U,V,W)$ if $A\in M(U,V,W)$. We write $\opsp(A)$ for the set of all limit operators of $A$. Similarly, $B=(b_{ij})_{i,j\in\Z}$ is a limit operator of $A_+=(a_{ij})_{i,j\in\N}\in\BDO(\ell^p(\N))$ if \eqref{eq:limop} holds for a sequence $(h_n)$ in $\N$ with $h_n\to+\infty$. Note that limit operators are always given by a bi-infinite matrix, no matter if the matrix $A$ or $A_{+}$ to start with is bi- or semi-infinite. The following lemma summarises the main results on limit operators:

\begin{lemma}\label{lem:limop}
Let $A\in\BDO(\ell^{p}(\I))$ with $\I\in \{\Z,\N\}$ and let $B$ be a limit operator of $A$. Then:\\[-7mm]
\begin{itemize}\itemsep-1mm
\item[{\bf a)}] \cite{RaRoSi1998} It holds that $\|B\|\le\|A\|$.
\item[{\bf b)}] \cite{RaRoSi1998,RaRoSiBook} If $A$ is Fredholm then $B$ is invertible, and $B^{-1}$ is a limit operator of any regularizer $C$ of $A$.
(Note that $B,B^{-1}\in\BDO(\ell^p(\Z))$ and $C\in\BDO(\ell^{p}(\I))$ hold if $A\in\BDO(\ell^{p}(\I))$.)
\item[{\bf c)}] \cite{RaRoSi1998,LiBook,CWLi2008:Memoir,LiSei:BigQuest} A is Fredholm iff all its limit operators are invertible.
\item[{\bf d)}] \cite{Davies2001:PseudoErg,LiBook} If $A\in M(U,V,W)$ or $A\in M_+(U,V,W)$ then $A$ is pseudoergodic iff $\opsp(A)=M(U,V,W)$.
\item[{\bf e)}] \cite{RaRoRoe} If $A\in\BDO(\ell^p(\N))$ is Fredholm then $B_+:=(B_{ij})_{i,j\in\N}$ is Fredholm and $\ind(A)=\ind(B_+)$.
\end{itemize}
\end{lemma}
So we immediately get that $A\in\PE(U,V,W)$ is Fredholm iff all $B\in M(U,V,W)$ are invertible, in which case of course $A\in M(U,V,W)$ is invertible.

{\bf Finite sections.} A further topic of \cite{LindnerRoch2010} and our paper is the so-called {\sl finite section method (FSM)}. This method aims to approximately solve an equation $Ax=b$, i.e.
\begin{equation} \label{eq:Ax=b}
\sum_{j\in\Z} a_{ij}\ x(j)\ =\ b(i),\quad i\in\Z,
\end{equation}
by truncating it to
\begin{equation} \label{eq:Anxn=b}
\sum_{l_n\le j\le r_n} a_{ij}\ x_n(j)\ =\ b(i),\quad l_n\le i\le r_n,
\end{equation}
where the cut-off points $l_1,l_2,...\to -\infty$ and $r_1,r_2,...\to+\infty$ are certain, sometimes well-chosen, integers.
The aim is that, assuming invertibility of $A$ (i.e. unique solvability of \eqref{eq:Ax=b} for all right-hand sides $b$), also \eqref{eq:Anxn=b} shall be uniquely solvable for all sufficiently large $n$ and the solutions $x_n$ shall approximate the solution $x$ of \eqref{eq:Ax=b} (in the sense that, for every right hand side $b$, it holds as $n\to\infty$ that $\|x_n-x\|\to 0$, if $1\leq p<\infty$, that $\|x_n\|=O(1)$ and $x_n(j)\to x(j)$ for every $j\in\Z$, if $p=\infty$).
If that is the case 
then the FSM is said to be {\sl applicable} to $A$ (or we say that the FSM {\em applies to} $A$).

If $i$ and $j$ in \eqref{eq:Ax=b} only run over the positive integers, $\N$, then this system corresponds to a semi-infinite equation $A_+x_+=b_+$. The FSM is then to freeze $l_n$ at $1$ and only let $r_n$ go to $+\infty$. Otherwise the terminology is identical.

Applicability of the FSM is equivalent \cite{ProssdorfSilbermann,RoSi,RaRoSiBook} to invertibility of $A$ plus {\sl stability} of the sequence of finite matrices
\begin{equation} \label{eq:An}
A_{n}\ :=\ (a_{ij})_{i,j=l_{n}}^{r_{n}},\qquad n=1,2,...\ .
\end{equation}
The latter means that, for all sufficiently large $n$, the matrices $A_{n}$ (the so-called {\sl finite sections} of $A$) are invertible and their inverses are uniformly bounded, in short: $\limsup_{n\to\infty} \|A_{n}^{-1}\|<\infty$. This, moreover, is known \cite{LindnerRoch2010,SeidelPhD} to be equivalent to the invertibility of $A$ and of certain semi-infinite matrices that are associated to $A$ and to the cut-off sequences $(l_{n})$ and $(r_{n})$. Those associated semi-infinite matrices are partial limits (in the strong topology) of the upper left and the lower right corner of the finite matrix $A_{n}$ as $n\to\infty$. Precisely, the associated matrices are the entrywise limits
\begin{equation} \label{eq:ass_matrix}
(a_{i+l_n',j+l_n'})_{i,j=0}^\infty\ \to\ B_+\qquad\textrm{and}\qquad (a_{i+r_n',j+r_n'})_{i,j=-\infty}^0\ \to\ C_-\qquad\textrm{as}\ n\to\infty
\end{equation}
of semi-infinite submatrices of $A$, where $(l_n')_{n=1}^\infty$ and $(r_n')_{n=1}^\infty$ are subsequences of $(l_n)_{n=1}^\infty$ and $(r_n)_{n=1}^\infty$, respectively, such that the limits \eqref{eq:ass_matrix} exist. So $B_{+}$ and $C_-$ are one-sided truncations of limit operators of $A$; they tell us what we find in the limit when jumping along the main diagonal of $A$ via the sequences $l_{1},l_{2},...$ and $r_{1},r_{2},...$ -- or subsequences thereof. Hence, by the choice of the cut-off sequences $(l_{n})$ and $(r_{n})$, one can control the selection of associated matrices $B_{+}$ and $C_-$ and consequently control the applicability of the FSM. Let us summarize all that:

\begin{lemma}\label{lem:FSMappl}
For $A=(a_{ij})_{i,j\in\Z}\in\BDO(\ell^p(\Z))$ and two cut-off sequences
$(l_n)_{n=1}^\infty$ and $(r_n)_{n=1}^\infty$ in $\Z$ with $l_n\to-\infty$ and $r_n\to+\infty$, the following are equivalent:\\[-7mm]
\begin{itemize}
\item[i)] the FSM \eqref{eq:Anxn=b} is applicable to $A$,\\[-7mm]
\item[ii)] the sequence $(A_n)_{n=1}^\infty$, with $A_n$ from \eqref{eq:An}, is stable,\\[-7mm]
\item[iii)] $A$ and all limits $B_+$ and $C_-$ from \eqref{eq:ass_matrix} are invertible.\\[-7mm]
\item[iv)] $A$ and all limits $B_+$ and $C_-$ from \eqref{eq:ass_matrix} are invertible, and the inverses $B_+^{-1}$ and $C_-^{-1}$ are uniformly bounded.
\end{itemize}
\end{lemma}

\begin{proof}
That applicability $i)$ is equivalent to invertibility of $A$ plus stability $ii)$ is a classical result (called ``Polski's theorem'' in \cite{HaRoSi2}) for the case of strong convergence $A_n\to A$, and it is in \cite{RoSi} for the more general case considered here. (Note that the convergence $A_n\to A$ is generally not strong if $p=\infty$.) That $ii)$ is equivalent to $iv)$ was shown in \cite{RaRoSi2001} for $p=2$, in \cite{RaRoSiBook} for $p\in [1,\infty)$ and in \cite{LiBook} for $p\in[1,\infty]$. So in particular, $ii)$ implies invertibility of $A$ and hence also $i)$. The equivalence of $iii)$ and $iv)$ is shown in \cite{RaRoSi:OnFS} for $p\in (1,\infty)$ and in \cite{Li:FSMsubs} for $p\in[1,\infty]$. The case of arbitrary monotonic cut-off sequences $(l_n)$ and $(r_n)$ can be found in \cite{SeidelPhD,LindnerRoch2010}.
\end{proof}

If the FSM is applicable to $A$ then, for every right hand side $b$, $x_n(j)\to x(j)$ as $n\to\infty$ for every $j\in \Z$, where $x_n$ is the solution to \eqref{eq:Anxn=b} and $x$ that of \eqref{eq:Ax=b}. But this implies that $\|x\|\leq \liminf_{n\to\infty} \|x_n\|$, for every $b$, and hence
\begin{equation} \label{eq:liminf}
\liminf_{n\to\infty} \|A_n^{-1}\|\ \ge\ \|A^{-1}\|;\ \quad \mbox{ and similarly } \quad \liminf_{n\to\infty} \|A_n\|\ \ge\ \|A\|,
\end{equation}
this latter holding whether or not the FSM is applicable to $A$. Complementing this bound, we remark that it can be deduced from \cite[Proposition 3.1]{SeidelSilb13} (or see \cite[Section 6]{HagLiSei}) that
\begin{equation} \label{eq:limsup}
\limsup_{n\to\infty}\|A_{n}^{-1}\| = \sup\{\|A^{-1}\|,\|B_{+}^{-1}\|,\|C_{-}^{-1}\|\},
 \end{equation}
 where this supremum is taken over all limits $B_+$ and $C_-$ in \eqref{eq:ass_matrix}, and is attained as a maximum if the FSM is applicable to $A$. In our arguments below we will not need \eqref{eq:limsup}, however, only the much simpler lower bound \eqref{eq:liminf}.

Versions of Lemma \ref{lem:FSMappl}, \eqref{eq:liminf}, and \eqref{eq:limsup} hold for semi-infinite matrices $A_+=(a_{ij})_{i,j\in\N}\in$ $\BDO(\ell^p(\N))$, with the modification that $l_n=1$, which implies that every limit $B_+$ in \eqref{eq:ass_matrix} is nothing but the matrix $A_+$ again, so that in Lemma \ref{lem:FSMappl} iii), iv) and \eqref{eq:limsup} it is the invertibility of only $A_+$ and $C_-$ that is at issue.

\begin{remark} {\bf -- Reflections.} \label{rem:reflect}
Often we find it convenient to rearrange/reflect the matrices $C_-=(c_{ij})_{i,j=-\infty}^0$ from \eqref{eq:ass_matrix} as $B_+=(c_{-j,-i})_{i,j=0}^\infty$.
This rearrangement $C_-\mapsto B_+$ corresponds to a matrix reflection against the bi-infinite antidiagonal; it can be written as $B_+=RC_-^\top R$, where $R$ denotes the bi-infinite flip $(x_i)_{i\in\Z}\mapsto (x_{-i})_{i\in\Z}$. As an operator on $\ell^p$, one gets $\|B_+\|_p=\|RC_-^\top R\|_p=\|C_-^\top\|_p=\|C_-\|_q$ with\footnote{Here and in what follows we put $\infty^{-1}:=0$, so that $p=1\Rightarrow q=\infty$ and $p=\infty\Rightarrow q=1$.} $p^{-1}+q^{-1}=1$.  When we  speak below about the FSM for $A$ and its ``associated semi-infinite submatrices $B_+$'' we will mean all $B_+$ from \eqref{eq:ass_matrix} plus the reflections $B_+=RC_-^\top R$ of all $C_-$ from \eqref{eq:ass_matrix}.
\end{remark}

A simple choice of cut-off sequences is to take $l_{n}=-n$ and $r_{n}=n$ for $n=1,2,...$. This is called the {\sl full FSM} for $A$. For a semi-infinite matrix $A_+$ the full FSM is to take $l_n= 1$ and $r_n=n$ for $n=1,2,...$. In either case, the full FSM leads to more associated matrices $B_{+}$ (and hence to a smaller chance for applicability of the FSM) than ``thinning out'' those cut-off sequences in a way that suits the matrix $A$ (or $A_+$) at hand. For example, if $A\in\PE(U,V,W)$ then all $B_{+}\in M_{+}(U,V,W)$ are associated to $A$ in case of the full FSM. So, in addition to $A$ itself, all $B_{+}\in M_{+}(U,V,W)$ have to be invertible to make sure the full FSM applies to $A$. That is why, in \cite{LindnerRoch2010}, the cut-offs $l_n$ and $r_n$ have been placed very sparsely and in a special way that leads to all associated $B_{+}$ being Toeplitz.  A simple consequence of Lemma \ref{lem:FSMappl} is the following lemma which also (trivially, by the equivalence of i) and ii) in Lemma \ref{lem:FSMappl}) holds for semi-infinite matrices $A_+=(a_{ij})_{i,j\in\N}$ (with $l_n=1$ in that case).

\begin{lemma} \label{lem:fullFSM}
If the full FSM applies to $A=(a_{ij})_{i,j\in\Z}\in\BDO(\ell^p(\Z))$ then the FSM with any  monotonic cut-off sequences $(l_n)_{n=1}^\infty$ and $(r_n)_{n=1}^\infty$ applies to $A$.
\end{lemma}
\begin{proof}
Suppose that the full FSM method is applicable to $A$. Fix two arbitrary monotonic cut-off sequences $(l_n)_{n=1}^\infty\to -\infty$ and $(r_n)_{n=1}^\infty \to+\infty$ and look at two associated matrices $B_+$ and $C_-$ from \eqref{eq:ass_matrix} with respect to subsequences $(l_n')$ and $(r_n')$ of $(l_n)$ and $(r_n)$. Since $(l_n)$ and $(r_n)$ are (at least for sufficiently large $n$) subsequences of the sequences $(-1,-2,-3,...)$ and $(1,2,3,...)$ that are used for the full FSM, the same is true for $(l_n')$ and $(r_n')$. So $B_+$ and $C_-$ are also associated to $A$ in case of the full FSM and so, by the equivalence of i) and iii) in Lemma \ref{lem:FSMappl}, are invertible together with $A$. Again by the equivalence of i) and iii) in Lemma \ref{lem:FSMappl}, since $A$ and all matrices $B_+$ and $C_-$ associated to these cut-off sequences are invertible, the FSM with cut-off sequences $(l_n)$ and $(r_n)$ applies to $A$.
\end{proof}

Lemma \ref{lem:fullFSM} is why we place particular focus on the full FSM: it is the most demanding version of the FSM -- if this version applies then all cut-off sequences will be fine.

\medskip

{\bf Main Results. } Having set the notations, let us now sketch our main results. For operators $A\in\PE(U,V,W)$ and $B_+\in\PE_+(U,V,W)$, we are interested in the four sets
\[
\spess A,\quad \spec A,\quad \spess B_+\quad\textrm{ and }\quad\spec B_+.
\]
From \cite{LindnerRoch2010} we know that the first three sets coincide,
\begin{equation} \label{eq:Sintro}
\spess A\ =\ \spec A\ =\ \spess B_+,
\end{equation}
and are independent of $A$ and $B_+$, as long as these are pseudoergodic. We will show that also the fourth set, $\spec B_+$, is independent of $B_+$, and we indicate what the difference between the two sets is. The key to describe the difference between $\spess B_+$ and $\spec B_+$ is a new result that has a famous cousin in the theory of Toeplitz operators: Coburn's Lemma \cite{Coburn} says that, for every bounded and nonzero Toeplitz operator $T_{+}$, one has $\alpha(T_{+})=0$ or $\beta(T_{+})=0$, so that $T_{+}$ is known to be invertible as soon as it is Fredholm and has index zero. We prove that the same statement holds with the Toeplitz operator $T_{+}$ replaced by any $B_+\in M_+(U,V,W)$ provided that $0$ is not in \eqref{eq:Sintro}. So $\spess B_+$ and $\spec B_+$ differ by the set of all $\lambda\in\C$ for which $B_+-\lambda I_+$ is Fredholm with a nonzero index. We give new upper and lower bounds on the sets $\spess B_+$ and $\spec B_+$, and we find easily computable sets $G$ that close the gap between the two, i.e., sets $G$ for which it holds  that $\spec B_+ = G\cup \spess B_+$.

On the other hand, knowledge about invertibility of semi-infinite matrices $B_+\in M_+(U,V,W)$ is all we need to study applicability of the FSM, so that our new Coburn-type result has immediate consequences for the applicability of the FSM (even the full FSM) to pseudoergodic operators. In \cite{LindnerRoch2010} the question of the applicability of the full FSM to an operator $A_+\in\PE_+(U,V,W)$ could not be settled (nor could, in \cite{LindnerRoch2010}, the applicability of the full FSM to $A\in \PE(U,V,W)$: for brevity we just focus on the semi-infinite case in this paragraph). Instead, the cut-off sequences for the FSM were chosen (``adapted to $A_+$'') in a way that made all associated semi-infinite matrices $C_-$ in \eqref{eq:ass_matrix} Toeplitz. Classical Coburn then implied invertibility of all these $C_-$, as a consequence of their being Fredholm of index zero, which holds as long as $A_+$ is invertible. 
Thanks to the new Coburn result, this ``adaptation'' twist is no longer needed. The full FSM can be seen to apply by exactly the same argument, but with the associated operators $C_-$ no longer required to be Toeplitz.

Perhaps one of the main messages of our paper is that operators in $\PE_+(U,V,W)$ (termed ``stochastic Toeplitz operators'' in \cite{TrefContEmb}) behave a lot like usual Toeplitz operators when it comes to
\begin{itemize}\itemsep-1mm
\item the gap between essential spectrum and spectrum (both enjoy a lemma of Coburn type), and
\item having an applicable FSM (in both cases, the FSM applies iff the operator is invertible).
\end{itemize}
Similar coincidences can be shown for operators in $\PE(U,V,W)$ (the ``stochastic Laurent operators'' in the terminology of \cite{TrefContEmb}) and usual Laurent operators.

In Section 3 of our paper we show that the full FSM applies to $A_+\in\PE_+(U,V,W)$, and automatically to all other operators in $M_+(U,V,W)$ and in $M(U,V,W)$, as soon as $A_+$ is invertible. Even more, we show that all matrices in $\Mfin(U,V,W)$ are invertible if $A_+$ is invertible, so that the truncated systems \eqref{eq:Anxn=b} are uniquely solvable for all $n\ge 1$ (as opposed to $n\ge n_0$ with an $n_0$ that nobody knows). If $A_+\in\PE_+(U,V,W)$ is not invertible but Fredholm with index $\kappa=\kappa(U,V,W)\ne 0$, then the full (or any) FSM cannot be applied to any $A\in M(U,V,W)$. We however show that shifting the system $Ax=b$ down by $\kappa$ rows leads to a system to which the full (and hence any) FSM applies.

In Section 4 we bound and compare the norms, and the norms of inverses, of bi-infinite, semi-infinite and finite Jacobi matrices over $(U,V,W)$. This, in particular, allows the study of the resolvent norms, and hence the pseudospectra, of these operators and matrices. For example we show for $A_+\in \PE_+(U,V,W)$ and $A\in \PE(U,V,W)$ that, analogously to the corresponding result for the spectrum, $\speps^p A_+ = \speps^p A \cup G$, for $\eps>0$ and $p=2$. Here $G$ is any of the sets discussed above that closes the gap between $\spec A=\spess A_+$ and $\spec A_+$. And we are able to make close connections between the pseudospectra of finite and infinite matrices, for example showing that the union of all finite matrix pseudospectra, $\cup_{F\in \Mfin(U,V,W)}\specn_\eps^p F$, coincides with $\speps^p A_+$, for $p=2$. Our results in this section are a substantial generalisation of results in a study \cite{CCL2} of spectra and pseudospectra of a particular pseudoergodic Jacobi operator, the Feinberg-Zee random hopping matrix ($U=W=\{\pm 1\}$, $V=\{0\}$). Our results on the
relation between norms of inverses (and hence
pseudospectra) of finite and infinite stochastic Toeplitz
matrices, in particular our results on the convergence of norms, condition numbers, and pseudospectra of finite sections to their infinite stochastic Toeplitz counterparts, reproduce, in the case that $U$, $V$, and $W$ are singletons, results for (classical) Toeplitz operators and matrices \cite{ReichelTref,Boe94,BoeGruSilb97}.

{\bf Main techniques.} Besides the limit operator techniques behind Lemmas \ref{lem:limop} and \ref{lem:FSMappl} that were the core of \cite{LindnerRoch2010}, our second main tool is a ``glueing technique'' -- see \eqref{eq:matB} and \eqref{eq:matB+-} -- that is used in the proofs of two of our main results, Theorems \ref{thmA} and \ref{th:FSM2} as well as in their quantitative versions, Propositions \ref{prop:M+} and \ref{prop:F1}. This latter technique often complements the earlier in terms of relating finite, semi- and bi-infinite matrices to each other.

\section{Spectra of Pseudoergodic Operators} \label{sec:spectra}
 We recall, as discussed in Section \ref{sec:intro},  that throughout the paper we consider the operators in $M(U,V,W)\supset \PE(U,V,W)$ and $M_+(U,V,W)$ $\supset \PE_+(U,V,W)$ as acting on $\ell^p(\Z)$ and $\ell^p(\N)$, respectively, for some $p\in [1,\infty]$, and that the invertibility, Fredholmness (and also the index if Fredholm) of these operators is independent of $p$.
In Theorem 2.1 of \cite{LindnerRoch2010} the following was shown:
\begin{proposition} \label{prop:i-vi}
The following statements are equivalent:\\[-7mm]
\begin{itemize}\itemsep-1mm
\item[(i)] one operator in $\PE(U,V,W)$ is Fredholm,
\item[(ii)] one operator in $\PE(U,V,W)$ is invertible,
\item[(iii)] all operators in $M(U,V,W)$ are Fredholm,
\item[(iv)] all operators in $M(U,V,W)$ are invertible, 
\item[(v)] one operator in $\PE_{+}(U,V,W)$ is Fredholm,
\item[(vi)] all operators in $M_{+}(U,V,W)$ are Fredholm.
\end{itemize}
\end{proposition}

All these equivalences follow quickly from Lemma \ref{lem:limop}. Since the occurrence of one (and hence all) of the properties $(i)$--$(vi)$ is obviously not a matter of a concrete operator but rather of the interplay between $U$, $V$ and $W$, we will call the triple $(U,V,W)$ {\sl compatible} if $(i)$--$(vi)$ hold. We will see below in Proposition \ref{prop:PE} that, if $(U,V,W)$ is compatible, then also the inverses in (iv) are uniformly bounded, in fact bounded above by $\|A^{-1}\|$ for any $A\in \PE(U,V,W)$.

The equivalence of $(i)$, $(ii)$, $(iv)$, $(v)$ and $(vi)$ can also be expressed as follows: for every $A\in\PE(U,V,W)$ and every $B_{+}\in\PE_{+}(U,V,W)$, it holds that
\begin{equation} \label{eq:S1}
\spess A\ =\ \spec A\ =\ \bigcup_{C\in M(U,V,W)} \spec C\ =\ \bigcup_{C_{+}\in M_{+}(U,V,W)} \spess C_{+}\ =\ \spess B_{+}.
\end{equation}
Let us denote the set \eqref{eq:S1} by $\S(U,V,W)$ since it clearly depends on  $U$, $V$ and $W$ only and not on the choice of $A$ or $B_{+}$. Then $(U,V,W)$ is compatible iff $0\not\in \S(U,V,W)$.

To get a first idea of the set $\S:=\S(U,V,W)$, let us look at simple lower and upper bounds on $\S$, our discussion here taken from \cite[Theorem 2.1 a)]{LindnerRoch2010}. A lower bound on
\begin{equation} \label{eq:S2}
\S\ =\ \bigcup_{C\in M(U,V,W)} \spec C
\end{equation}
is clearly found by taking this union over a set of simple matrices $C\in M(U,V,W)$ for which $\spec C$ is known explicitly or is easily computed. Natural candidates are matrices $C$ with periodic diagonals, and the simplest among those are matrices with constant diagonals -- so-called Laurent (or bi-infinite Toeplitz) matrices.

If $C$ is the (only) element of $M(\{u\},\{v\},\{w\})$ with some $u\in U$, $v\in V$ and $w\in W$, i.e.~$C$ is the tridiagonal Laurent matrix with $u$, $v$ and $w$ on its diagonals, then \cite{BoeGru,BoeSi2}
\begin{equation} \label{eq:defE}
\spec C\ =\ \{ut+v+wt^{-1}\ :\ t\in\T\}\ =:\ E(u,v,w)
\end{equation}
is the ellipse depicted in Figure \ref{fig:ellipse}.a below. Note that
\[
E(u,v,w)\ =\ v\,+\,E(u,w)\qquad\textrm{with}\qquad E(u,w)\ :=\ E(u,0,w).
\]
Also note that \eqref{eq:defE} gives the ellipse $E(u,v,w)$ an orientation, based on the counter-clockwise orientation of the unit circle $\T$: the ellipse is oriented counter-clockwise if $|u|>|w|$, clockwise if $|u|<|w|$, and collapses into a line segment if $|u|=|w|$. From \eqref{eq:S2} and $\spec C=E(u,v,w)$ if $C$ is Laurent, we get that the union of all ellipses $E(u,v,w)$ with $u\in U$, $v\in V$ and $w\in W$ is a simple lower bound on $\S$:
\begin{equation} \label{eq:lowerS}
\bigcup_{u\in U, v\in V, w\in W} E(u,v,w)\ \subset\ \S(U,V,W).
\end{equation}
Because we will come back to Laurent and Toeplitz operators, let us from now on write
\[
T(u,v,w)\ :=\ uS\,+vI\,+wS^{-1}\qquad\textrm{and}\qquad T_+(u,v,w)
\]
for the Laurent operator $T\in M(\{u\},\{v\},\{w\})$, acting on $\ell^p(\Z)$, and for its compression $T_+$ to $\ell^p(\N)$, which is a Toeplitz operator. Here we write $S$ for the forward shift operator, $S:x\mapsto y$ with $y(j+1)=x(j)$ for all $j\in\Z$, and $S^{-1}$ for the backward shift. From \eqref{eq:S1} and \eqref{eq:defE} (or \cite{BoeGru,BoeSi2}), $\spess T_+(u,v,w)=\spec T(u,v,w)=E(u,v,w)$. Further,
\begin{equation} \label{eq:specT+}
\spec T_+(u,v,w)\ =\ \conv E(u,v,w)
\end{equation}
is the same ellipse but now filled \cite{BoeGru,BoeSi2}. (Here $\conv S$ denotes the convex hull of a set $S\subset \C$.)  Let $\Ein(u,v,w)$ and $\Eout(u,v,w)$ denote the interior and exterior, respectively, of the ellipse $E(u,v,w)$, with the understanding that $\Ein(u,v,w)=\varnothing$ and $\Eout(u,v,w)=\C\setminus E(u,v,w)$ when $|u|=|w|$ and the ellipse $E(u,v,w)$ degenerates to a straight line. The reason why the spectrum of a Toeplitz operator $T_+$ is obtained from the spectrum of the Laurent operator $T$ (which is at the same time the essential spectrum of both $T_+$ and $T$) by filling in the hole $\Ein(u,v,w)$ can be found in the classical Coburn lemma \cite{Coburn}. We will carry that fact over to stochastic Toeplitz and Laurent operators.
A key role will also be played by the following index formula.
Let $\wind(\Gamma,z)$ denote the winding number (counter-clockwise) of an oriented closed curve $\Gamma$ with respect to a point $z\not\in \Gamma$. For $0\not\in E(u,v,w)$, so that $T_+$ is Fredholm, it holds that \cite{BoeGru,BoeSi2}
\begin{equation} \label{rq:index}
\ind T_+(u,v,w)\ =\ -\wind(E(u,v,w),0)\ =\ \left\{\begin{array}{ll}
                                                0,\ & 0\ \in\ \Eout(u,v,w), \\
                                                \sign(|w|-|u|),\ & 0\ \in\ \Ein(u,v,w).
                                              \end{array}\right.
\end{equation}

To get a simple upper bound on $\S$, write $A\in\PE(U,V,W)$ as $A=D+T$ with diagonal part $D=\diag(v_{i})$ and off-diagonal part $T$ and think of $A$ as a perturbation of $D$ by $T$ with $||T||\le\eps$, where $\eps:=u^{*}+w^{*}$ and
\[
u^{*}\ :=\ \max_{u\in U}|u|, \qquad w^{*}\ :=\ \max_{w\in W}|w|.
\]
Since $A$ is in the $\eps$-neighbourhood of $D$, its spectrum $\spec A=\S$ is in the $\eps$-neighbourhood of $\spec D\subset V$. (Note that $D$ is normal or look at Lemma 3.3 in \cite{LindnerRoch2010}.) In short,
\begin{equation} \label{eq:upperS}
\S(U,V,W)\ \subset\ V+(u^{*}+w^{*})\ovD
\end{equation}
(recall that $\D:=\{z\in\C:|z|<1\}$ is the open unit disk, and $\ovD$ is its closure). Note that the same argument, and hence the same upper bound, applies to the spectra of all (singly or bi-)infinite and all finite Jacobi matrices over $U$, $V$ and $W$.

Sometimes equality holds in \eqref{eq:upperS} but often it does not. For $U=\{1\}$, $V=\{0\}$ and $W=\T$, the lower \eqref{eq:lowerS} and upper bound \eqref{eq:upperS} on $\S$ coincide so that equality holds in \eqref{eq:upperS} saying that $\S=2\ovD$. If we change $W$ from $\T$ to $\{-1,1\}$ then the right-hand side of \eqref{eq:upperS} remains at $2\ovD$ while $\S$ is now smaller (it is properly contained in the square with corners $\pm 2$ and $\pm 2\ri$, see \cite{CCL2,CWDavies2011}). Taking $W$ even down to just $\{1\}$, the spectrum $\S$ clearly shrinks to $[-2,2]$ with the right-hand side of \eqref{eq:upperS} still at $2\ovD$. So the gap in \eqref{eq:upperS} can be considerable, or nothing, or anything in between, really.

Equality \eqref{eq:S1} contains the formula
\[
\spess B_{+}\ =\ \bigcup_{C_{+}\in M_{+}(U,V,W)} \spess C_{+}\ =\ \S
\]
for all $B_{+}\in\PE_{+}(U,V,W)$. One of our new results, Corollary \ref{cor:specB+} below, is that
\begin{equation} \label{eq:S+}
\spec B_{+}\ =\ \bigcup_{C_{+}\in M_{+}(U,V,W)} \spec C_{+}\ =:\ \S_{+}
\end{equation}
holds independently of $B_{+}\in\PE_{+}(U,V,W)$.

Upper and lower bounds on $\S_{+}=\S_{+}(U,V,W)$ can be derived in the same way as above for $\S$. This time, because of \eqref{eq:specT+}, the ellipses in the lower bound \eqref{eq:lowerS} have to be filled in, while the upper bound from \eqref{eq:upperS} remains the same, so that
\begin{equation} \label{eq:lowerupperS+}
\bigcup_{u\in U, v\in V, w\in W} \conv E(u,v,w)\ \subset\ \S_{+}(U,V,W)\ \subset\ V+(u^{*}+w^{*})\ovD.
\end{equation}
The results in this section will also make precise the difference between $\S$ and $\S_{+}$.

For nonzero Toeplitz operators $T_{+}$ (semi-infinite matrices with constant diagonals), acting boundedly on $\ell^{p}(\N)$, the following classical result fills the gap between essential spectrum and spectrum: at least one of the two integers, $\alpha(T_{+})$ and $\beta(T_{+})$, is always zero. So if their difference is zero (i.e.~$T_{+}$ is Fredholm with index zero) then both numbers are zero (i.e.~$T_{+}$ is injective and surjective, hence invertible). This is Coburn's Lemma \cite{Coburn}, which was also found, some years earlier, by Gohberg \cite{Goh58} (but for the special case of Toeplitz operators with continuous symbol). Here is a new cousin of that more than 50 year old lemma:
\begin{theorem} \label{thmA}
If $(U,V,W)$ is compatible (i.e.~(i)--(vi) hold in Proposition \ref{prop:i-vi}) then every $B_+\in M_+(U,V,W)$ is Fredholm and at least one of the non-negative integers $\alpha(B_+)$ and $\beta(B_+)$ is zero.
\end{theorem}
\begin{proof}
Let $(U,V,W)$ be compatible and take $B_+\in M_+(U,V,W)$ arbitrarily. Then $B_+$, with matrix representation $(\tilde b_{ij})_{i,j\in\N}$, is Fredholm since $(i)$--$(vi)$ of Proposition \ref{prop:i-vi} hold. Suppose that $\alpha(B_+)>0$ and $\beta(B_+)>0$. Then there exist
$x\in \ell^p(\N)$ and $y\in \ell^q(\N)$, with $p^{-1}+q^{-1}=1$, $x\neq 0$, and $y\neq 0$, such that $B_+x = 0$ and $B_+^\top y=0$.
 Let $a,b\in \C$, define $z\in \ell^\infty(\Z)$ by
\[
z = (\cdots, ay_2, ay_1,\boxed{0},bx_1,bx_2,\cdots)^\top,
\]
where the box marks the entry $z_0$,
and define $B\in M(U,V,W)$ by its matrix representation $(b_{ij})_{i,j\in\Z}$ with
\[
b_{ij} = \left\{\begin{array}{cc}
                  \tilde b_{ij}, & i,j\in \N, \\
                  0, & |i-j|>1, \\
                  \tilde b_{-j,-i}, & i,j\in -\N.
                \end{array}
\right.
\]
The remaining entries $b_{0,-1},b_{1,0}\in U$, $b_{0,0}\in V$ and $b_{0,1},b_{-1,0}\in W$ of $B$ may be arbitrary.
Then $B_+x=0$ and $B_+^\top y=0$ together imply that $(Bz)_j = 0$ for all $j\neq 0$. Further,
\[
(Bz)_0 = b_{0,-1}ay_1 + b_{0,1} b x_1.
\]
Clearly we can pick $a,b\in \C$ with $a\neq 0$ or $b\neq 0$ (so that $z\neq 0$) to ensure that also $(Bz)_0=0$ which implies that $Bz=0$, so that $B$ is not invertible on $\ell^\infty(\Z)$. But this is a contradiction since $B\in M(U,V,W)$ is invertible by our assumption that the triple $(U,V,W)$ is compatible.
\end{proof}

An immediate corollary of \cite[Theorem 2.1]{LindnerRoch2010} or Proposition \ref{prop:i-vi} and our Theorem \ref{thmA} is the following Coburn lemma for pseudoergodic operators. This result reduces to the usual Coburn lemma for Toeplitz operators (at least to the special case in which the Toeplitz operator is tridiagonal) in the case that $U$, $V$, and $W$ are singleton sets.

\begin{corollary}\label{corB}
If $B_+\in \PE_+(U,V,W)$ is Fredholm then at least one of $\alpha(B_+)$ and $\beta(B_+)$ is zero.
\end{corollary}
\begin{proof}
If an operator in $\PE_{+}(U,V,W)$ is Fredholm then $(U,V,W)$ is compatible by Proposition \ref{prop:i-vi}. Now apply Theorem \ref{thmA} to $B_+$.
\end{proof}

Similarly to the situation for Toeplitz operators, one can now derive invertibility of operators in $M_+(U,V,W)$ from their Fredholmness and index. The additional result here that every $B_+\in M_+(U,V,W)$ is Fredholm with the same index was first pointed out in \cite{LindnerRoch2010}, as a consequence of the main result from \cite{RaRoRoe} (see our Lemma  \ref{lem:limop} e)), and will play an important role in our arguments.

\begin{corollary} \label{corC}
Let $(U,V,W)$ be compatible. Then every $B_+\in M_+(U,V,W)$ is Fredholm with the same index
\[
\kappa(U,V,W)\ :=\ \ind B_+\ \in\ \{-1,0,1\}.
\]
Further,
if $\ind B_+ =0$ then $B_+$ is invertible; if $\ind B_+=1$ then $\alpha(B_+)=1$ and $\beta(B_+)=0$; and if $\ind B_+=-1$ then $\alpha(B_+)=0$ and $\beta(B_+)=1$. These statements are independent of the choice of $B_+\in M_+(U,V,W)$.
\end{corollary}
\begin{proof}
If the triple $(U,V,W)$ is compatible then every $B_+\in M_+(U,V,W)$ is Fredholm, by Proposition \ref{prop:i-vi}. To see that all $B_+\in M_+(U,V,W)$ have the same index, apply Lemma \ref{lem:limop} e) to an $A_+\in\PE_+(U,V,W)$. Because $M_+(U,V,W)$ includes the Toeplitz operators $T_+(u,v,w)$ with $(u,v,w)\in (U,V,W)$, it follows from \eqref{rq:index} that $\kappa(U,V,W)\in\{-1,0,1\}$. The remaining claims follow from Theorem \ref{thmA}.
\end{proof}

\begin{corollary} \label{cor:specB+}
An operator $B_+\in\PE_+(U,V,W)$ is invertible iff all operators $C_+\in M_+(U,V,W)$ are invertible. In other words, for all $B_{+}\in\PE_{+}(U,V,W)$,
\[
\spec B_{+}\ =\ \bigcup_{C_{+}\in M_{+}(U,V,W)} \spec C_{+},
\]
this set denoted $\S_+=\S_+(U,V,W)$ in \eqref{eq:S+}.
\end{corollary}
\begin{proof}
By Corollary \ref{corC} if one operator $B_+\in\PE_+(U,V,W)$ is invertible then $(U,V,W)$ is compatible, $\kappa(U,V,W)=0$, and all operators in $M_+(U,V,W)$ are invertible. The formula for the spectrum follows by considering $B_+-\lambda I_+$ instead of $B_+$.
\end{proof}

\noindent As an extension to this corollary we will see in Theorem \ref{th:FSM1} that, if $B_+\in\PE_+(U,V,W)$ is invertible, then
\begin{equation} \label{eq:EPplusinv}
\sup_{C_+  \in M_+(U,V,W)} \|C_+^{-1}\| < \infty.
\end{equation}

Before we explore further what the above results mean for spectra, we note some consequences of Corollary \ref{corC} that are essentially captured in \cite[Theorem 2.4]{LindnerRoch2010}. If $(U,V,W)$ is compatible then Corollary \ref{corC} tells us that every $B_+\in M_+(U,V,W)$ has the same index $\kappa(U,V,W)$. In particular $\kappa = \ind T_+(u,v,w)$, for every $(u,v,w)\in (U,V,W)$, this index given by \eqref{rq:index}. It follows that, if $(U,V,W)$ is compatible, then this index must have the same value for all $(u,v,w)\in (U,V,W)$, so that either $0\in \Eout(u,v,w)$ for all $(u,v,w)\in (U,V,W)$, in which case $\kappa=0$, or $0\in \Ein(u,v,w)$ and $|w|-|u|$ has the same sign for all $(u,v,w)\in (U,V,W)$, in which case $\kappa = -\wind(E(u,v,w),0)= \sign(|w|-|u|)$. Thus, if $A_+\in \PE_+(U,V,W)$ is Fredholm but not invertible, either
\begin{equation} \label{eq:ind2}
w_*\ :=\ \min_{w\in W} |w|\ >\ u^*\ \mbox{ (when $\kappa=1$) }\; \mbox{ or }\; u_*\ :=\ \min_{u\in U} |u|\ >\ w^*\ \mbox{ (when $\kappa=-1$).}
\end{equation}
In the first case 0 is circumnavigated clockwise by all ellipses $E(u,v,w)$, in the second the circumnavigation is counter-clockwise.

For $\lambda\in\C$, put $V-\lambda:=\{v-\lambda:v\in V\}$ and note that $B_+-\lambda I_+\in M_+(U,V-\lambda,W)$ iff $B_+\in M_+(U,V,W)$. Similarly, $B_+-\lambda I_+\in\PE_+(U,V-\lambda,W)$ iff $B_+\in\PE_+(U,V,W)$. We split the complex plane $\C$ into four pairwise disjoint parts. To this end, fix an arbitrary $B_+\in\PE_+(U,V,W)$. The first part of the plane is our set $\S=\S(U,V,W)$ from \eqref{eq:S1},
\begin{eqnarray*}
\S &=& \{\lambda\in\C\ :\ B_+-\lambda I_+\ \textrm{is not Fredholm} \}\\
&=& \spess B_+\ =\ \bigcup_{C_+\in M_+(U,V,W)} \spess C_+\\
&=& \{\lambda\in\C\ :\ (U,V-\lambda,W) \textrm{ is not compatible} \}.
\end{eqnarray*}
The rest of the complex plane now splits into the following three parts: for $k=-1,0,1$, let
\[
\S_k\ :=\ \{\lambda\in\C\ :\ \ind(B_+-\lambda I_+)=k\}\
=\ \{\lambda\in\C\ :\ \kappa(U,V-\lambda,W)=k\}.
\]
So we have a partition (i.e.~a splitting into pairwise disjoint sets) of $\C$:
\begin{equation} \label{eq:splitS}
\C\ =\ \underbrace{\S\ \cup\ \S_{-1}\ \cup\ \S_1}_{\S_+}\ \cup\ \S_0,
\end{equation}
where the equality
\begin{equation} \label{eq:S+2}
\S_+\ =\ \S\,\cup\,\S_{-1}\,\cup\,\S_1\ =\ \C\,\setminus\,\S_0
\end{equation}
holds by Corollary \ref{corC}. So the difference between $\S_+=\spec B_+$ from \eqref{eq:S+} and $\S=\spess B_+$ from \eqref{eq:S1} is precisely $\S_{-1}\cup\S_1$.

The computation of these four parts of the plane, $\S$, $\S_{-1}$, $\S_1$ and $\S_0$, is of course far from trivial (otherwise spectral theory of random Jacobi operators would be easy) but we will compare this partition of $\C$ with another partition of $\C$ that is closely related and, in contrast, easy to compute. To do this, let $\E$ denote the set of all ellipses $E(u,v,w)$ with $u\in U$, $v\in V$ and $w\in W$, and put

\noindent
\begin{center}
\includegraphics[width=\textwidth]{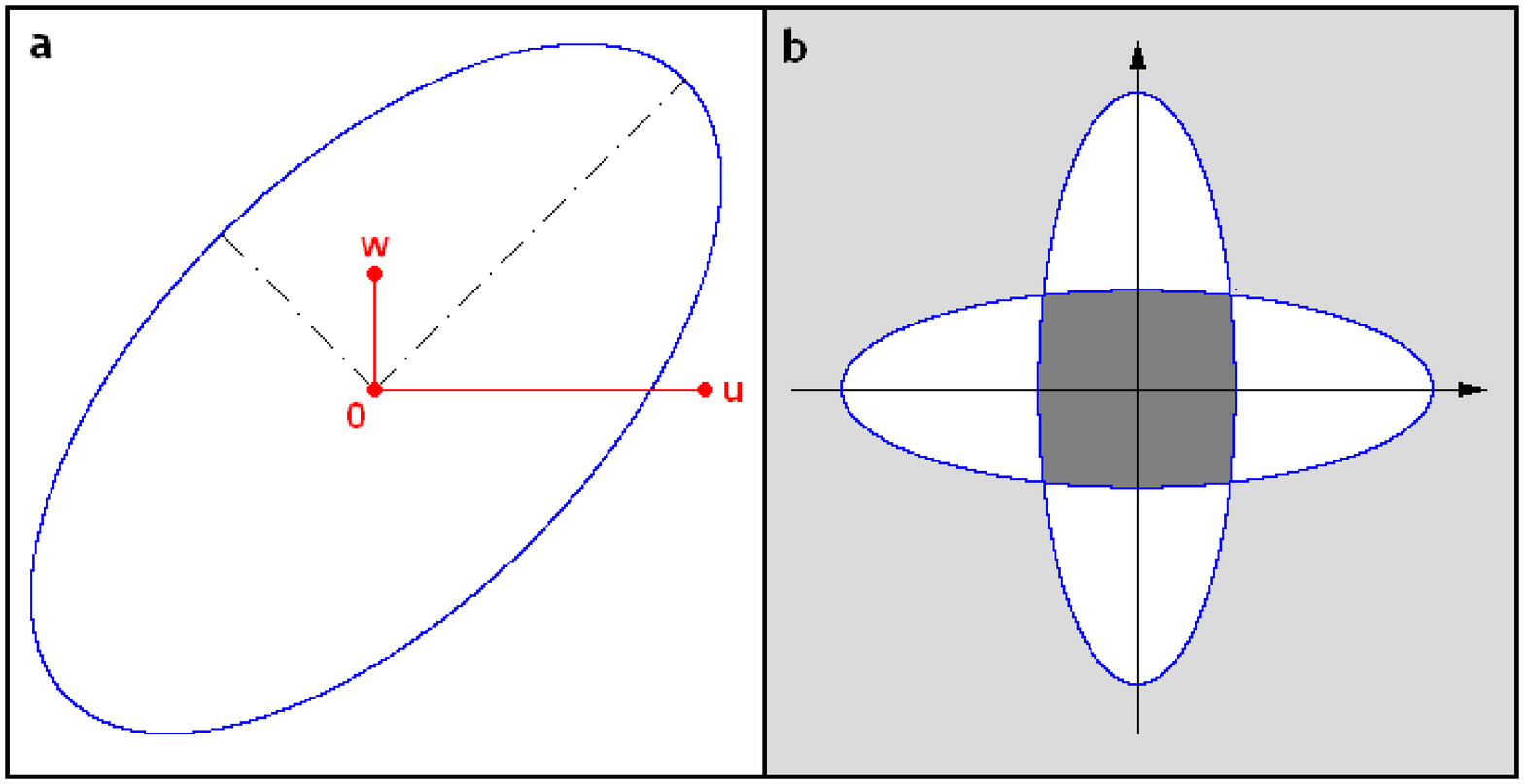}
\end{center}
~\\[-17mm]  
\begin{figure}[h]
\caption{\footnotesize a) The ellipse $E(u,v,w)$ with general values $u,v,w\in\C$ is derived from the zero-centered ellipse $E(u,w):=E(u,0,w)$ after translation by $v$. The ellipse $E(u,w)$ is centered at the origin and has orthogonal half-axes (the dotted lines) of length $\big||u|\pm|w|\big|$, respectively, where the major axis bisects the angle between $u$ and $w$ at the origin. Depicted here is the ellipse $E(u,w)$ for the particular values $u=3$ and $w=\ri$.\newline
b) We see the splitting of $\C$ into the four parts $E$, $E_{-1}$, $E_0$ and $E_1$ for
$U=\{-1,1\}$, $V=\{0\}$ and $W=\{2\}$: the dark gray area is $E_1$, the light gray area is $E_0$ and the rest (the white area plus the ellipse boundaries) is $E$. In this example, $E_{-1}$ is empty. Note that both ellipses, $E(-1,2)$ and $E(1,2)$, are oriented clockwise.
} \label{fig:ellipse}
\end{figure}

\begin{equation}\label{eqs:E}
\left.
\begin{aligned}
E_0 &\ :=\  \{\lambda\in\C\ :\ \lambda \textrm{ is outside of all ellipses in }\E\},\\
\Eins &\ :=\  \{\lambda\in\C\ :\ \lambda \textrm{ is inside all ellipses in }\E\},\\
E_1&\ :=\  \{\lambda\in\Eins\ :\ \lambda \textrm{ is circumnavigated clockwise by all ellipses in }\E\},\\
E_{-1} &\ :=\  \{\lambda\in\Eins\ :\ \lambda \textrm{ is circumnavigated counter-clockwise by all ellipses in }\E\},\\
E &\ :=\  \C\setminus(E_0\cup E_1\cup E_{-1}),\\
\Eall &\ :=\  \C\setminus E_0\ =\ \bigcup_{u\in U, v\in V, w\in W} \conv E(u,v,w).
\end{aligned}
~ \qquad ~ \right\}
\end{equation}
\noindent Obviously, at most one of the sets $E_{1}$ and $E_{-1}$ is nonempty. The set $E$ consists of the points that lie on one of the ellipses, or they are inside some but outside other ellipses, or they are inside all ellipses but circumnavigated clockwise by some and counter-clockwise by others. We have
\begin{equation} \label{eq:splitE}
\C\ =\ E\ \cup\ E_{-1}\ \cup\ E_1\ \cup\ E_0,
\end{equation}
in analogy to \eqref{eq:splitS}. But while the ingredients of \eqref{eq:splitS} are in general notoriously difficult to compute, those of \eqref{eq:splitE} are easily drawn. Before we relate \eqref{eq:splitS} to \eqref{eq:splitE}, it is perhaps time for an example.

\begin{example} \label{ex:CWDav}
Take $U=\{-1,1\}$, $V=\{0\}$ and $W=\{2\}$. Neither $\S$ nor $\S_+$ is precisely known in this case (but see \cite{CWDavies2011,Hagger:NumRange} for bounds on both). But the ingredients of \eqref{eq:splitE} are easy to write down: draw all ellipses $E(u,v,w)$ with $u\in U$, $v\in V$ and $w\in W$. In this case, there are only $|U|\cdot|V|\cdot|W|=2$ ellipses: $E(-1,0,2)$ and $E(1,0,2)$. The situation is depicted in Figure \ref{fig:ellipse}.b (which is taken from \cite{LindnerRoch2010}).
The dark gray area is our set $\Eins$, the light gray area is $E_0$, and the rest (the white area plus the ellipses themselves) is $E$. Both ellipses are oriented clockwise since $|u|<|w|$ in both cases. So in this example, $E_1=\Eins$ and $E_{-1}$ is empty.
\end{example}

Now let us come to the relation between the partitions \eqref{eq:splitS} and \eqref{eq:splitE}. From the discussion above \eqref{eq:ind2} we see that 
\begin{equation} \label{eq:Sk_in_Ek}
\S_{-1}\subset E_{-1},\qquad \S_{1}\subset E_{1},\qquad \S_{0}\subset E_{0}
\qquad\textrm{and hence}\qquad E\subset\S.
\end{equation}
So we have at least some simple upper bounds on $\S_{-1}$, $\S_1$ and $\S_0$ and a lower bound on $\S$. The upper bound on $\S_0$ is equivalent to the lower bound $\Eall$ on its complement $\S_+$ in \eqref{eq:lowerupperS+}. The lower bound $E$ on $\S$ is actually sharper than the lower bound \eqref{eq:lowerS}. Further, from the discussion leading to \eqref{eq:ind2} we see that
\begin{equation} \label{eq:EE}
E_1\ =\ \left\{\begin{array}{cc}
               E_\cap,\ & \mbox{if } w_*\ >\ u^*, \\
               \varnothing,\ & \mbox{otherwise,}
             \end{array}\right.
  \qquad
E_{-1}\ =\ \left\{\begin{array}{cc}
               E_\cap,\ & \mbox{if } u_*\ >\ w^*, \\
               \varnothing,\ & \mbox{otherwise.}
             \end{array}\right.
\end{equation}

Recall that the difference between $\S_+$ and $\S$ (i.e.~the non-essential spectrum of $B_+$) is $\S_{-1}\cup\S_1$. From \eqref{eq:Sk_in_Ek} we get that
\begin{equation} \label{eq:Spm}
\S_{\pm 1}\ :=\ \S_{-1}\cup\S_1\ \subset\ E_{-1}\cup E_1\ =: E_{\pm 1}\ \subset\ \Eins.
\end{equation}
So the non-essential spectrum of $B_+\in\PE_+(U,V,W)$ is inside all ellipses in $\E$.
By \eqref{eq:lowerupperS+}, we have
\begin{equation} \label{eq:Eins}
\Eins\ \subset\ \Eall\ \subset\ \S_{+}.
\end{equation}
Now let us sum up. From \eqref{eq:S+2}, \eqref{eq:Spm} and \eqref{eq:Eins} we get
\[
\S_+\ =\ \S\cup\S_{\pm 1}\ \subset\ \S\cup E_{\pm 1}\ \subset\ \S\cup\Eins\ \subset\ \S\cup\Eall\ \subset\  \S\cup\S_+\ =\ \S_+,
\]
so that all inclusions are in fact equalities and we have proven the following:

\begin{theorem} \label{thm:gap}
It holds that
\begin{equation} \label{eq:S+-S}
\S_+\ =\ \S\,\cup\,\S_{\pm 1}\ =\ \S\,\cup\,E_{\pm 1}\ =\ \S\,\cup\,\Eins\ =\ \S\,\cup\,\Eall
\end{equation}
with $\S_{\pm 1}$ and $E_{\pm 1}$ from \eqref{eq:Spm} and $\Eins$ and $\Eall$ from \eqref{eqs:E}. Moreover, exactly one of the following cases applies:
\begin{description}
  \item[i)] $w_*\le u^*$ and $u_* \le w^*$, in which case $\S_{\pm 1} = E_{\pm 1} = \varnothing$ and $E_\cup \subset \S_+ = \S$;
  \item[ii)]  $w_*> u^*$, in which case $\S_{-1} = E_{-1} = \varnothing$ and $\S_1 \subset E_{1} = E_\cap$;
  \item[iii)] $u_* > w^*$, in which case $\S_{1} = E_{1} = \varnothing$ and $\S_{-1} \subset E_{-1} = E_\cap$.
\end{description}
\end{theorem}
Of course $\Eins$ (and certainly $\Eall$) is in general larger than the actual gap $\S_+\setminus\S=\S_{\pm 1}$ between the spectrum and essential spectrum, but equality \eqref{eq:S+-S} is still an attractive new bit of the picture: we do not know, very explicitly, what the sets $\S$ from \eqref{eq:S1} and $\S_+$ from \eqref{eq:S+} are, but we do now know explicitly what we have to add on to $\S$ to get $\S_+$.
It has also recently been shown \cite{Hagger:NumRange} 
 that $\convn(\Eall)=\convn(E)$ is, surprisingly, both the closure of the numerical range and the convex hull of the spectrum for each operator in $\PE(U,V,W)\cup\PE_+(U,V,W)$, and is hence a (very explicit) upper bound on both $\S$ and $\S_{+}$. Combining this result with Theorem \ref{thm:gap} and \eqref{eq:Sk_in_Ek} we have that
\begin{equation} \label{eq:hagger}
E\ \subset\ \S\ \subset\ \S_{+}\ =\ \S \cup \Eall\ \subset\ \convn(\Eall)\ =\ \convn(E).
\end{equation}
\begin{example} \label{ex:bidiag}
{\bf ~--~Bidiagonal case. } In \cite{LiBiDiag} the bidiagonal case was studied, that means $U=\{0\}$ or $W=\{0\}$. Let us say $U=\{0\}$. Then all our ellipses
\[
E(u,v,w)\ =\ E(0,v,w)\ =\ v+w\T,\qquad u\in U,\ v\in V,\ w\in W
\]
are circles with clockwise orientation. So we have that
\[
\Eall\ =\ \bigcup_{v\in V} (v+w^*\ovD)
\qquad\textrm{and}\qquad
\Eins\ =\ \bigcap_{v\in V} (v+w_*\D)
\]
with $w^*=\max_{w\in W}|w|$ and $w_*=\min_{w\in W}|w|$.
In \cite{LiBiDiag} it was shown that
\[
\S\ =\ \Eall\setminus\Eins
\qquad\textrm{and hence, by \eqref{eq:S+-S}, we have}\qquad
\S_+\ =\ \Eall.
\]
So in this case, the partitions \eqref{eq:splitS} and \eqref{eq:splitE} coincide:
\[
\S=E=\Eall\setminus\Eins,\quad \S_{-1}=E_{-1}=\varnothing,\quad
\S_1=E_1=\Eins,\quad \S_0=E_0=\C\setminus\Eall.
\]
For $\lambda\in\S_+\setminus\S=\S_1=\Eins$, each $B_+\in\PE_+(U,V-\lambda,W)$ is Fredholm with index $1$; precisely, $\alpha(B_+)=1$ and $\beta(B_+)=0$.
\end{example}

\begin{example} \label{ex:possi} Take $U=\{1\}$, $V=\{0\}$, and $W=\{0,2\}$. In this case there are two ellipses: $E(1,0) = \T$, with a counter-clockwise orientation, and
\begin{equation} \label{eq:elli12}
E(1,2)\ =\ \{t_r+i t_i:t_r, t_i\in \R \mbox{ and } (t_r/3)^2 + t_i^2 = 1\},
\end{equation}
with a clockwise orientation. Case i) in Theorem \ref{thm:gap} applies, so that $E_\cup \subset \S_+=\S$. Further, $E(1,0)\subset \conv E(1,2)$, so that $E_\cup = \conv E_\cup$ and, by \eqref{eq:hagger}, $\S_+=\S = E_\cup = \conv E(1,2)$.
\end{example}

\section{The Finite Section Method} \label{sec:FSM}
From Theorems 2.8 and 2.9 of \cite{LindnerRoch2010}, or see Lemma \ref{lem:FSMappl} and the comments immediately below that lemma (or \cite{SeidelPhD}), we know that the full FSM (with $l_n=1$ and $r_n=n$) applies to a semi-infinite Jacobi matrix $A_+$ iff the operator itself and the set of associated semi-infinite matrices $C_-$ in \eqref{eq:ass_matrix} are invertible. To each matrix $C_{-}$ corresponds (see Remark \ref{rem:reflect}) a reflected matrix $C_+=RC_-^\top R\in M_+(U,V,W)$. Further, the set of all these reflected matrices $C_+$ is all of $M_{+}(U,V,W)$ iff $A_+$ is pseudoergodic, as a simple consequence of Lemma \ref{lem:limop}d). Similarly, the full FSM (with $l_n=-n$ and $r_n=n$) applies to a bi-infinite Jacobi matrix $A$ iff the operator itself and both sets of associated semi-infinite matrices $B_+$ and $C_-$ in \eqref{eq:ass_matrix} are invertible and, again by Lemma \ref{lem:limop}d), the union of the set of all matrices $B_+$ and $C_+=RC_-^\top R$ is the whole of $M_{+}(U,V,W)$ iff $A$ is pseudoergodic.  As a simple consequence of these facts, and the results in sections \ref{sec:intro} and \ref{sec:spectra}, we obtain:

\begin{theorem} \label{th:FSM1}
The following are equivalent:
\begin{itemize}
\item[(a)] the full FSM applies to one operator in $\PE(U,V,W)$;\\[-4ex]
\item[(b)] the full FSM applies to one operator in $\PE_{+}(U,V,W)$;\\[-4ex]
\item[(c)] all operators in $M_{+}(U,V,W)$ are invertible;\\[-4ex]
\item[(d)] all operators in $M_{+}(U,V,W)$ are invertible, and the inverses are uniformly bounded;\\[-4ex]
\item[(e)] one operator in $\PE_{+}(U,V,W)$ is invertible;\\[-4ex]
\item[(f)] $0\not\in \S_{+}$;\\[-4ex]
\item[(g)] one operator in $\PE_{+}(U,V,W)$ is Fredholm with index $0$;\\[-4ex]
\item[(h)] all operators in $M_{+}(U,V,W)$ are Fredholm with index $0$;\\[-4ex]
\item[(i)] $(U,V,W)$ is compatible and $\kappa(U,V,W)=0$;\\[-4ex]
\item[(j)] the full FSM applies to all operators in $M(U,V,W)$;\\[-4ex]
\item[(k)] the full FSM applies to all operators in $M_{+}(U,V,W)$.
\end{itemize}
\end{theorem}
\begin{proof}
The equivalence of (c), (g), (h) and (i) follows from Proposition \ref{prop:i-vi} and Corollary \ref{corC}. The equivalence of (a)-(c) is then clear from  the remarks preceding this theorem and Lemma \ref{lem:FSMappl}, as is the equivalence of (c) with (j) and (k) (or see Theorems 2.8 and 2.9 in \cite{LindnerRoch2010}). That (c), (e) and (f) are equivalent is Corollary \ref{cor:specB+}. To see the equivalence of (c) and (d) suppose that (c) holds, in which case also (j) holds. Then, by Lemma \ref{lem:FSMappl}, all the operators $B_+$ and $C_-$ in \eqref{eq:ass_matrix} arising from the full FSM applied to any $A\in M(U,V,W)$ are invertible and uniformly bounded. Thus, for all $p\in [1,\infty]$,
\begin{equation} \nonumber
\sup \|B_+^{-1}\|_p < \infty \;\mbox{ and }\; \sup\|C_-^{-1}\|_p < \infty,
\end{equation}
so that also, where $C_+=RC_-^\top R$ and since $\ell^q(-\N)$ is the dual space of $\ell^p(-\N)$ if $p^{-1}+q^{-1}=1$, $\sup\|C_+^{-1}\|_q<\infty$. Since, see the remarks before the lemma, the collection of all operators $B_+$ and $C_+$ is the whole of $M_+(U,V,W)$ in the case that $A\in \PE(U,V,W)$, we see that we have shown (d), precisely that
\begin{equation} \label{eq:PEinvp}
\cN_{+,p} := \sup_{B_+\in M(U,V,W)}\|B_+^{-1}\|_p < \infty,
\end{equation}
for $1\leq p\leq \infty$.
\end{proof}

\begin{remark} \label{rem:FSM}
{\bf a) } So for pseudoergodic semi-infinite matrices $A_+$ (the so-called ``stochastic Toeplitz operators'' from \cite{TrefContEmb}), we get that the full FSM applies as soon as the operator is invertible. This is of course the best possible result since invertibility of $A_+$ is a minimal requirement (it is necessary) for the applicability of the FSM. For (classical) banded Toeplitz operators $T_+$, the same is true as was first shown in \cite{GohbergFeldman}. So in a sense, we also rediscover that classical result for tridiagonal Toeplitz matrices (by applying our result to the case when $U$, $V$ and $W$ are singletons).

{\bf b) } There is a similar coincidence between the FSM for pseudoergodic bi-infinite matrices (called ``stochastic Laurent operators'' in \cite{TrefContEmb}) and for usual Laurent operators: in both cases, the FSM applies iff the operator is invertible and the corresponding semi-infinite principal submatrix (the Toeplitz part) has index zero. If the latter index is not zero, there is something that can be done. It is called ``index cancellation'', and we will get to this in short course.

{\bf c) } Recall from Lemma \ref{lem:fullFSM} (also see \cite[Thms.~2.8 \& 2.9]{LindnerRoch2010}) that the FSM applies with arbitrary monotonic cut-off sequences $l_{n}$ and $r_{n}$ if the full FSM is applicable.
\end{remark}

Recall that applicability of the FSM means that the truncated (finite) systems \eqref{eq:Anxn=b} are uniquely solvable for all sufficiently large $n$, say $n\ge n_{0}$, with their solutions $x_{n}$ approximating the unique solution $x$ of \eqref{eq:Ax=b}. The limit operator techniques behind Lemma \ref{lem:FSMappl} don't reveal much about that -- practically very relevant -- number $n_{0}$; but in our pseudoergodic setting we can prove that actually
\begin{center}
$n_{0}=1$ holds in $(a)$, $(b)$, $(j)$ and $(k)$ of Theorem \ref{th:FSM1}:
\end{center}

\begin{theorem} \label{th:FSM2}
From the equivalent conditions (a)--(k) of Theorem \ref{th:FSM1} it follows that all finite Jacobi matrices over $U$, $V$ and $W$,
that means all $F\in\Mfin(U,V,W)$, are invertible.
\end{theorem}

Before we come to the proof, let us note that this theorem implies
\begin{equation} \label{eq:specJ}
\spec J\ \subset\ \S_+
\end{equation}
for every Jacobi matrix $J$ (finite or infinite) over $U$, $V$ and $W$, i.e., for every $J\in\Mfin(U,V,W)\cup M_+(U,V,W)\cup M(U,V,W)$.
Equality holds if $J\in\PE_+(U,V,W)$.

We now prepare the proof of Theorem \ref{th:FSM2}. It combines a technique from the proof of \cite[Theorem 4.1]{CCL2} with
elements of the proof of Theorem \ref{thmA}. Let $n\in\N$. Given an $n\times n$ matrix $F\in\Mfin(U,V,W)$
and arbitrary elements $u\in U$, $v\in V$ and $w\in W$, we make the following construction. Put
\begin{equation}\label{eq:matB}
B\ :=\
\left(\begin{array}{ccccccc}
\ddots&
\begin{array}{ccc}w&& \end{array}\\
\cline{2-2}
\begin{array}{c}u\\ \\ \\ \end{array}&
\multicolumn{1}{|c|}{F}&
\begin{array}{c} \\ \\w \end{array}\\
\cline{2-2}
&\begin{array}{ccr}&&u\end{array}
& \boxed{v} &
\begin{array}{lcc}w&&\end{array}\\
\cline{4-4}
& & \begin{array}{c}u\\ \\ \\ \end{array}&
\multicolumn{1}{|c|}{{F}}&
\begin{array}{c} \\ \\w \end{array}\\
\cline{4-4}
&&&\begin{array}{ccr}&&u\end{array}
& v & \begin{array}{lcc}w&&\end{array}\\
\cline{6-6}
& && & \begin{array}{c}u\\ \\ \\ \end{array}&
\multicolumn{1}{|c|}{F}&
\begin{array}{c} \\ \\w \end{array}\\
\cline{6-6}
&&&&&
\begin{array}{ccc}&&u\end{array}&
\ddots
\end{array}\right)\ \in \ M(U,V,W),
\end{equation}
where $\boxed{v}$ marks the entry of $B$ at position $(0,0)$. We denote the semi-infinite blocks above and below $\boxed{v}$ by $B_{-}$ and $B_{+}$, respectively, so that
\begin{equation}\label{eq:matB+-}
B\ =\
\left(\begin{array}{ccc}
\multicolumn{1}{c|}{B_-}&
\begin{array}{c} \\ \\w \end{array}\\
\cline{1-1}
\begin{array}{ccr}&&u\end{array}
& \boxed{v} & \begin{array}{lcc}w&&\end{array}\\
\cline{3-3}
 & \begin{array}{c}u\\ \\ \\ \end{array}&
\multicolumn{1}{|c}{{B_+}}
\end{array}\right).
\end{equation}
Precisely, with $B=(b_{ij})_{i,j\in\Z}$, we put $B_+:=(b_{ij})_{i,j\in\N}\in M_+(U,V,W)$ and $B_-:=(b_{ij})_{i,j\in-\N}$.

Now, for a vector $x\in \C^{n}$ and a complex sequence $(r_{k})_{k\in \Z}$, put
\begin{equation}\label{eq:vecx}
\widetilde x\ :=\ \left(\begin{array}{c}
\vdots\\
\multicolumn{1}{|c|}{~}\\
\hline
0\\
\hline
\multicolumn{1}{|c|}{~}\\
\multicolumn{1}{|c|}{r_{-1}x}\\
\multicolumn{1}{|c|}{~}\\
\hline\\[-2.4ex]
\boxed{0}\\[0.4ex]
\hline
\multicolumn{1}{|c|}{~}\\
\multicolumn{1}{|c|}{{r_0x}}\\
\multicolumn{1}{|c|}{~}\\
\hline
0\\
\hline
\multicolumn{1}{|c|}{~}\\
\multicolumn{1}{|c|}{r_1x}\\
\multicolumn{1}{|c|}{~}\\
\hline
0\\
\hline
\multicolumn{1}{|c|}{~}\\
\vdots
\end{array}
\right),
\qquad\textrm{leading to}\qquad
B\widetilde x\ =\
\left(\begin{array}{c}
\vdots\\
\multicolumn{1}{|c|}{~}\\
\hline
z_{-1}\\
\hline
\multicolumn{1}{|c|}{~}\\
\multicolumn{1}{|c|}{r_{-1}Fx}\\
\multicolumn{1}{|c|}{~}\\
\hline\\[-2.4ex]
\boxed{z_{0}}\\[0.8ex]
\hline
\multicolumn{1}{|c|}{~}\\
\multicolumn{1}{|c|}{{r_0Fx}}\\
\multicolumn{1}{|c|}{~}\\
\hline
z_{1}\\
\hline
\multicolumn{1}{|c|}{~}\\
\multicolumn{1}{|c|}{r_1Fx}\\
\multicolumn{1}{|c|}{~}\\
\hline
z_{2}\\
\hline
\multicolumn{1}{|c|}{~}\\
\vdots
\end{array}
\right),
\end{equation}
where $\boxed{0}$ and $\boxed{z_{0}}$ mark the respective $0$ positions and
\begin{equation} \label{eq:zn}
z_{k}\ =\ r_{k-1}ux_{n}+r_{k}wx_{1},\qquad k\in\Z.
\end{equation}

\begin{lemma}\label{lem:choose-rk}
For arbitrary $x_1,x_n,u,w\in\C$, there exists a sequence $(r_k)_{k\in\Z}$ in $\C$ so that, for the sequence $(z_k)_{k\in\Z}$ from \eqref{eq:zn}, either
\begin{equation}\label{eq:r+}
(r_0,r_1,...)\in \ell^\infty(\N)\setminus\{0\}
\qquad \textrm{and}\qquad z_1=z_2=...=0
\end{equation}
or
\begin{equation}\label{eq:r-}
(...,r_{-2},r_{-1})\in \ell^\infty(-\N)\setminus\{0\}
\qquad\textrm{and}\qquad z_{-1}=z_{-2}=...=0.
\end{equation}
\end{lemma}
\begin{proof}
The choice of $(r_k)$ differs, depending on whether (and which) parameters are zero. If none of $x_1,x_n,u,w$ is zero, put
$\rho := -ux_n/(wx_1)$ and $r_k := \rho^k$ for $k\in \Z$, so that $z_k=0$ for $k\in \Z$ and \eqref{eq:r+} holds if $|\rho|\le 1$, \eqref{eq:r-} if $|\rho|\ge 1$. If $ux_n=0$, put $r_0:=1$ and $r_k:=0$ for $k\neq 0$, in which case \eqref{eq:r+} holds.  If $wx_1=0$, put $r_{-1}:=1$ and $r_k:=0$ for $k\neq -1$, in which case  \eqref{eq:r-} holds.
\end{proof}

Now we have all that we need:

\begin{proofof}{Theorem \ref{th:FSM2}}
Let $F\in M_n(U,V,W)$ for some $n\in\N$.
 Suppose $F$ is singular. Then there is an $x\in\C^n\setminus\{0\}$ with $Fx=0$. Choose a sequence $(r_k)$ in $\C$ so that \eqref{eq:r+} or \eqref{eq:r-} holds. Then, in the notations of \eqref{eq:matB+-} and \eqref{eq:vecx}, either $B_+\tilde x_+=0$ with $\tilde x_+:=\tilde x|_\N\in\ell^\infty(\N)$ or $B_-\tilde x_-=0$ with $\tilde x_-:=\tilde x|_{-\N}\in\ell^\infty(-\N)$. In either case, there is a non-invertible operator (either $B_+$ or the reflection $RB_-^\top R$ of $B_-$) in $M_+(U,V,W)$, which contradicts our assumption. 
So $F$ is invertible.
\end{proofof}

\medskip

So, in the case $\kappa(U,V,W)=0$, the full FSM (and hence the FSM with any monotonic cut-off sequences $l_n$ and $r_n$) applies to every $A\in M(U,V,W)$ and to every $A_+\in M_+(U,V,W)$, where unique solvability of the finite systems \eqref{eq:Anxn=b} already starts at $n=1$.

In the remaining cases (assuming, of course, that $(U,V,W)$ is compatible, i.e.~$\kappa(U,V,W)$ is defined), it holds that
\[
\kappa\ :=\ \kappa(U,V,W)\ =\ \pm 1,
\]
and the FSM cannot apply -- no matter how the cut-offs are placed (e.g. \cite[Prop 5.2]{Li:FSMsubs}). In the semi-infinite case the operator $A_+$ is not even invertible since its index equals $\kappa\ne 0$. In the bi-infinite case, the way out is to move the system $Ax=b$ up or down by one row (i.e.~to renumber the infinitely many equations in \eqref{eq:Ax=b} by increasing or decreasing the row number $i$ by $1$). So instead of $Ax=b$, the equivalent system $S^\kappa Ax=S^\kappa b$ is solved, where $S$ is the bi-infinite forward shift introduced earlier. Passing from $A$ to $\tilde A:=S^{\kappa}A$ preserves invertibility and corrects the index of the semi-infinite principal submatrices from $\kappa$ to $0$. Indeed,
\begin{equation} \label{eq:cancel}
\ind \tilde A_{+}\ =\ \ind(S^{\kappa}A)_{+}\ =\ \ind (S^{\kappa})_{+}+\ind A_{+}\ =\ -\kappa+\kappa\ =\ 0.
\end{equation}
This shifting process is called {\sl index cancellation}; for Laurent operators $A$ it goes back to \cite{GohbergFeldman,HeinigHellinger}, for much more general operators, see e.g. \cite{Li:FSMsubs,LindnerRoch2010,LiStrang:Perm,SeidelPhD}.

We claim that, after index cancellation,  the full FSM applies to every $A\in M(U,V,W)$ also in the cases $\kappa=\pm 1$. This can be seen as follows:

{\bf Case 1: } $\kappa=+1$.\quad Then all ellipses $E(u,v,w)$ with $u\in U$, $v\in V$ and $w\in W$ are oriented clockwise, so that
\begin{equation} \label{eq:w*>0}
|w|\ >\ |u|\ \ge\ 0,\qquad u\in U,\ w\in W.
\end{equation}
Now pass from $Ax=b$, i.e.
\begin{equation} \label{eq:noshift}
\left(\begin{array}{ccccccc} \ddots&\ddots\\[-2mm]
\ddots&v_{-2}&w_{-2}\\
&u_{-1}&v_{-1}&w_{-1}\\
\hline &&u_{0}&v_0&w_0\\\hline
&&&u_{1}&v_1&w_1\\[-1mm]
&&&&u_2&v_2&\ddots\\
&&&&&\ddots&\ddots
\end{array}\right)
\left(\begin{array}{c} \vdots\\x(-2)\\x(-1)\\x(0)\\x(1)\\x(2)\\
\vdots\end{array}\right) \ =\ \left(\begin{array}{c} \vdots\\b(-2)\\b(-1)\\\hline b(0)\\ \hline b(1)\\b(2)\\
\vdots\end{array}\right)
\end{equation}
to the equivalent system $SAx=Sb$, i.e.
\begin{equation} \label{eq:shift+1}
\left(\begin{array}{ccccccc} \ddots\\[-2mm]
\ddots&w_{-3}\\[-2mm]
\ddots&v_{-2}&w_{-2}\\
&u_{-1}&v_{-1}&w_{-1}\\
\hline &&u_{0}&v_0&w_0\\ \hline
&&&u_{1}&v_1&w_1\\
&&&&\ddots&\ddots&\ddots
\end{array}\right)
\left(\begin{array}{c} \vdots\\x(-2)\\x(-1)\\x(0)\\x(1)\\x(2)\\
\vdots\end{array}\right) \ =\ \left(\begin{array}{c} \vdots\\b(-3)\\b(-2)\\b(-1)\\ \hline b(0)\\ \hline b(1)\\
\vdots\end{array}\right).
\end{equation}
To see that the full FSM applies to the shifted system \eqref{eq:shift+1}, it is sufficient, by Theorem 2.8 of \cite{LindnerRoch2010}, to show that all semi-infinite matrices of the form
\begin{equation} \label{eq:B++1}
B_+\ =\ \left(\begin{array}{ccccc}
\tilde w_{1}\\
\tilde v_{2}&\tilde w_{2}\\
\tilde u_{3}&\tilde v_{3}&\tilde w_{3}\\
&\tilde u_{4}&\tilde v_4&\tilde w_4\\
&&\ddots&\ddots&\ddots
\end{array}\right)
\end{equation}
with $\tilde u_i\in U$, $\tilde v_i\in V$ and $\tilde w_i\in W$ are invertible on $\ell^p(\N)$. So let $B_{+}$ be one of them. We start with injectivity: from $\tilde w_i\ne 0$ for all $i$, by \eqref{eq:w*>0}, we get successively $x(1)=0,\ x(2)=0, \dots$ as the only solution of $B_{+}x_{+}=0$. By \cite{RaRoRoe}, $\ind B_{+}=\ind (S^{\kappa}A)_{+}$. But the latter is zero by \eqref{eq:cancel}, so that $B_{+}$ is also surjective and hence invertible.

{\bf Case 2: } $\kappa=-1$.\quad Now all ellipses $E(u,v,w)$ with $u\in U$, $v\in V$ and $w\in W$ are oriented counter-clockwise, so that
\begin{equation} \label{eq:u*>0}
|u|\ >\ |w|\ \ge\ 0,\qquad u\in U,\ w\in W.
\end{equation}
Now pass from $Ax=b$, i.e.~\eqref{eq:noshift}, to $S^{-1}Ax=S^{-1}b$, i.e.
\begin{equation} \label{eq:shift-1}
\left(\begin{array}{ccccccc}
\ddots&\ddots&\ddots\\
&u_{-1}&v_{-1}&w_{-1}\\
\hline &&u_{0}&v_0&w_0\\ \hline
&&&u_{1}&v_1&w_1\\[-1.5mm]
&&&&u_2&v_2&\ddots\\[-1.5mm]
&&&&&u_3&\ddots\\[-1.5mm]
&&&&&&\ddots
\end{array}\right)
\left(\begin{array}{c} \vdots\\x(-2)\\x(-1)\\x(0)\\x(1)\\x(2)\\
\vdots\end{array}\right) \ =\ \left(\begin{array}{c} \vdots\\b(-1)\\ \hline b(0)\\ \hline b(1)\\b(2)\\
b(3)\\\vdots\end{array}\right)
\end{equation}
and check in a similar way, now using \eqref{eq:u*>0} and \eqref{eq:cancel}, that all semi-infinite matrices of the form
\[
B_+\ =\ \left(\begin{array}{ccccc}
\tilde u_{1}&\tilde v_{1}&\tilde w_{1}\\
&\tilde u_{2}&\tilde v_2&\tilde w_2\\[-2mm]
&&\tilde u_{3}&\tilde v_3&\ddots\\[-2mm]
&&&\tilde u_{4}&\ddots\\[-2mm]
&&&&\ddots
\end{array}\right)
\]
with $\tilde u_i\in U$, $\tilde v_i\in V$ and $\tilde w_i\in W$ are invertible on $\ell^p(\N)$.

So also in case $\kappa(U,V,W)=\pm 1$, the FSM applies, with arbitrary cut-off sequences $l_n$ and $r_n$, to every $A\in M(U,V,W)$ -- but only after index cancellation. From \eqref{eq:w*>0} and \eqref{eq:u*>0} it is clear that every finite principal submatrix of the shifted (and therefore triangular) matrices in \eqref{eq:shift+1} and \eqref{eq:shift-1}, respectively, is invertible. So, again, the finite systems \eqref{eq:Anxn=b} are uniquely solvable for all $n$ (and not just for all sufficiently large $n$).

\section{Norms, Norms of Inverses, and Pseudospectra}
 In this section we bound and compare the norms and the norms of inverses of bi-infinite, semi-infinite and finite Jacobi matrices over $(U,V,W)$. The results are then expressed in terms of pseudospectra.

\subsection{Bi-infinite matrices}
We start with the simplest case: bi-infinite matrices in $M(U,V,W)$. Not surprisingly, a prominent role is played by those in $\PE(U,V,W)$.

\begin{proposition} \label{prop:PE}
Let $A\in\PE(U,V,W)$ and $B\in M(U,V,W)$ be arbitrary.
Then $\|A\|\ge\|B\|$. If moreover $(U,V,W)$ is compatible, then also $\|A^{-1}\|\ge\|B^{-1}\|$.
\end{proposition}
\begin{proof}
First note that $A$ and $B$ are invertible by Proposition \ref{prop:i-vi} if $(U,V,W)$ is compatible.
Now use Lemma \ref{lem:limop} d), b) and a) -- in this order. $A\in\PE(U,V,W)$ implies $B\in\opsp(A)$ by d), which then implies $B^{-1}\in\opsp(A^{-1})$ by b), so that $\|B\|\le\|A\|$ and $\|B^{-1}\|\le\|A^{-1}\|$ follow by a).
\end{proof}

\begin{corollary} \label{cor:PE}
For all $A\in\PE(U,V,W)$, we have
\[
\|A\| \ =\ \max\limits_{B\in M(U,V,W)}\|B\| \ =:\ \M.
\]
If moreover $(U,V,W)$ is compatible then
\begin{equation}\label{eq:||PE||}
\|A^{-1}\|\ =\ \max\limits_{B\in M(U,V,W)}\|B^{-1}\| \ =:\ \cN 
\end{equation}
\end{corollary}
If we have a particular $p\in[1,\infty]$ in mind, or want to emphasise the dependence on $p$, we will write $\M_p$ and $\cN_p$ for the expressions $\M$ and $\cN$ defined in Corollary \ref{cor:PE} (cf.~\eqref{eq:PEinvp}).

The following proposition is a simple consequence of the observations that, if $A,B\in \BDO(\ell^p(\Z))$ for all $1\le p\le \infty$, and $A=RB^\top R$, where $R$ is the reflection operator defined in Remark \ref{rem:reflect}, then: (i) $\|A\|_p = \|B^\top \|_p = \|B\|_q$, for $1\le p \le \infty$, if $p^{-1}+q^{-1}=1$; (ii) $A\in M(U,V,W)$ iff $B\in M(U,V,W)$; (iii) $A$ is invertible iff $B$ is invertible, and if they are both invertible then $\|A^{-1}\|_p = \|(B^\top )^{-1}\|_p = \|B^{-1}\|_q$.

\begin{proposition} \label{prop:PEpq}
For $p,q\in[1,\infty]$, with
$p^{-1}+q^{-1}=1$, we have
\[
\M_p\ =\ \M_q\ \qquad\textrm{and}\qquad \cN_p = \cN_q. 
\]
\end{proposition}
\begin{proof}
It is clear from the above observations that $\M_p = \sup_{B\in M(U,V,W)} \|B^\top \|_p = \M_q$. Similarly $\cN_p = \cN_q$.
\end{proof}

\subsection{The relationship between semi- and bi-infinite matrices}\label{sec:M+vsM}
The semi-infinite case $M_+(U,V,W)$ is a bit more involved. We start with a simple observation:

\begin{proposition}\label{prop:PE+}
{\bf a)} For $A_+\in\PE_+(U,V,W)$ it holds that
\[
\|A_+\|\  =\ \max_{B_+\in M_+(U,V,W)}\|B_+\|\ =\ \M.
\]
\indent{\bf b)} If $A_+\in\PE_+(U,V,W)$ is invertible, i.e. $(U,V,W)$ is compatible and $\kappa(U,V,W)=0$, then $\|A_+^{-1}\|\ge \cN$.
\end{proposition}
\begin{proof}
Let $A\in M(U,V,W)$ be a bi-infinite extension of $A_+\in\PE_+(U,V,W)$. Then $A\in\PE(U,V,W)$, so that $\|A\|=\M$ and $\|A^{-1}\|=\cN$, by \eqref{eq:||PE||}. Now $\|A\|\ge\|A_+\|$ since $A_+$ is a compression of $A$, and $\|A\|\le\|A_+\|$ since $A\in\opsp(A_+)=M(U,V,W)$. If $A_+$ is invertible, this last fact implies that $A^{-1}\in\opsp(A_+^{-1})$, so that $\|A^{-1}\|\le\|A_+^{-1}\|$. Finally, let $B\in M(U,V,W)$ be a bi-infinite extension of a given arbitrary $B_+\in M_+(U,V,W)$. Then $B\in\opsp(A_+)=M(U,V,W)$ and $B_+$ is a compression of $B$, so that $\|B_+\|\le\|B\|\le\|A_+\|$.
\end{proof}

\noindent
A question we have been unable to resolve in general is whether $\|A_+^{-1}\|$ is the same for all $A_+\in\PE_+(U,V,W)$, and whether it is larger than $\cN$. We will see below that the following proposition, which is a partial complement of the bound $\|A_+^{-1}\|\ge\cN$ for $A_+\in \PE_+(U,V,W)$, answers this question at least in the case $p=2$. The arguments to obtain this proposition are a quantitative version of the proof of Theorem \ref{thmA}.
\begin{proposition}\label{prop:M+}
{\bf a) } If $A_+,B_+\in M_+(U,V,W)$ are invertible and $(U,V,W)$ is compatible, then
\begin{equation}\label{eq:M+min}
\min\left(\|A_+^{-1}\|_p,\|B_+^{-1}\|_q\right)\ \le\ \cN_p\ =\ \cN_q
\end{equation}
holds for all $p,q\in[1,\infty]$ with $p^{-1}+q^{-1}=1$.

{\bf b) } Thus, if $(U,V,W)$ is compatible with $\kappa(U,V,W)=0$, so that all operators in $M_+(U,V,W)$ are invertible, then \eqref{eq:M+min} holds for all $A_+,B_+\in M_+(U,V,W)$.

{\bf c) } If $A_+,B_+\in \PE_+(U,V,W)$ are invertible then equality holds in \eqref{eq:M+min}.
\end{proposition}
\begin{proof}
{\bf a) } If $(U,V,W)$ is compatible then all operators in $M(U,V,W)$ are invertible and $\cN_p=\cN_q$ holds by Proposition \ref{prop:PEpq}. Let $A_+,B_+\in M_+(U,V,W)$ be invertible, and let $C_-:=RB_+^\top R$ be the reflection of $B_+$ as discussed in Remark \ref{rem:reflect}. Abbreviate
\[
\|A_+^{-1}\|_p=:a\quad\textrm{and}\quad\|C_-^{-1}\|_p=\|(RB_+^\top R)^{-1}\|_p=\|R(B_+^\top)^{-1} R\|_p=\|(B_+^{-1})^\top\|_p=\|B_+^{-1}\|_q=:b.
\]
Given an arbitrary $\eps>0$, choose $x\in\ell^p(\N)$ and $y\in\ell^p(-\N)$ so that $\|A_+x\|_p=1$, $\|C_-y\|_p=1$ with $\|x\|_p^p>a^p-\eps$ and $\|y\|_p^p>b^p-\eps$ in case $p<\infty$, and $\|x\|_\infty>a-\eps$ and $\|y\|_\infty>b-\eps$ in case $p=\infty$. Now put
\begin{equation} \label{eq:Bdef}
B\ :=\
\left(\begin{array}{ccc}
\multicolumn{1}{c|}{C_-}&
\begin{array}{c} \\ \\w \end{array}\\
\cline{1-1}
\begin{array}{ccr}&&u\end{array}
& \boxed{v} & \begin{array}{lcc}w&&\end{array}\\
\cline{3-3}
 & \begin{array}{c}u\\ \\ \\ \end{array}&
\multicolumn{1}{|c}{{A_+}}
\end{array}\right)\ \in\ M(U,V,W)
\quad\textrm{and}\quad
\widetilde x\ :=\ \left(\begin{array}{c}
\multicolumn{1}{|c|}{\vdots}\\
\multicolumn{1}{|c|}{sy}\\
\multicolumn{1}{|c|}{~}\\
\hline\\[-2.4ex]
\boxed{0}\\[0.4ex]
\hline
\multicolumn{1}{|c|}{~}\\
\multicolumn{1}{|c|}{{rx}}\\
\multicolumn{1}{|c|}{\vdots}
\end{array}
\right)\ \in\ \ell^p(\Z),
\end{equation}
where $u\in U$, $v\in V$ and $w\in W$ are chosen arbitrarily, $\boxed{v}$ marks the entry at $(0,0)$ in $B$, $\boxed{0}$ is at position $0$ in $\widetilde x$, and $r,s\in\C$ are chosen such that $(B\widetilde x)_0=usy_{-1}+wrx_1$ equals zero, while $r,s$ are not both zero. Then $\widetilde x\ne 0$ and $B\widetilde x=(sC_-y,0,rA_+x)^\top\ne 0$. Now, for $p<\infty$, 
\begin{align*}
\|B^{-1}\|_p^p&\ge\frac{\|B^{-1}B\widetilde x\|_p^p}{\|B\widetilde x\|_p^p}=\frac{\|\widetilde x\|_p^p}{\|B\widetilde x\|_p^p}=\frac{|r|^p\|x\|_p^p+|s|^p\|y\|_p^p}{|r|^p\|A_+x\|_p^p+|s|^p\|C_-y\|_p^p}>\frac{|r|^pa^p+|s|^pb^p}{|r|^p+|s|^p}-\eps\\
&=: ta^p+(1-t)b^p-\eps \ge \min(a,b)^p-\eps
\end{align*}
holds, where $t:=|r|^p/(|r|^p+|s|^p)\in[0,1]$. Since this is the case for every $\eps>0$, we conclude $\|B^{-1}\|_p\ge\min(a,b)$. The case $p=\infty$ is similar:
\begin{align*}
\|B^{-1}\|_\infty&\ge\frac{\|B^{-1}B\widetilde x\|_\infty}{\|B\widetilde x\|_\infty}=\frac{\|\widetilde x\|_\infty}{\|B\widetilde x\|_\infty}=\frac{\max(|r|\|x\|_\infty,|s|\|y\|_\infty)}{\max(|r|\|A_+x\|_\infty,|s|\|C_-y\|_\infty)}\\
&>\frac{\max(|r|(a-\eps),|s|(b-\eps))}{\max(|r|,|s|)}=\frac{\max(|r|a,|s|b)}{\max(|r|,|s|)}-\eps \ge \min(a,b)-\eps
\end{align*}
Since $\eps>0$ is arbitrary, again $\|B^{-1}\|_\infty\ge\min(a,b)$ follows.

{\bf b) } If, in addition, $\kappa(U,V,W)=0$ then all operators in $M_+(U,V,W)$ are invertible, so that {\bf a)} can be applied to arbitrary $A_+,B_+\in M_+(U,V,W)$.

{\bf c) } If $A_+,B_+\in\PE_+(U,V,W)$ are invertible then the conditions of {\bf b)} are satisfied, by Theorem \ref{th:FSM1}.
In addition to {\bf b)}, note that, by Proposition \ref{prop:PE+} b), the minimum in \eqref{eq:M+min} is greater than or equal to $\min(\cN_p,\cN_q)$ if $A_+,B_+\in\PE_+(U,V,W)$, which equals $\cN_p=\cN_q$, by Proposition \ref{prop:PEpq}.
\end{proof}

We have seen already in Theorem \ref{th:FSM1} that $\cN_{+,p}$, our notation \eqref{eq:PEinvp} for the supremum of $\|C_+^{-1}\|_p$ over all $C_+\in M(U,V,W)$, is finite for $1\le p \le \infty$ if $(U,V,W)$ is compatible with $\kappa(U,V,W)=0$.  Propositions \ref{prop:PE+} b) and \ref{prop:M+} imply that, in many cases, this supremum is in fact a maximum and coincides with the maximum \eqref{eq:||PE||} and with the norm of $A_+^{-1}$ when $A_+\in \PE_+(U,V,W)$.

\begin{proposition} \label{prop:scenarios} Suppose that $(U,V,W)$ is compatible with $\kappa(U,V,W)=0$, equivalently that all operators in $M_+(U,V,W)$ are invertible. Then, for every $p,q\in [1,\infty]$ with $p^{-1}+q^{-1}=1$:

{\bf a)}

\vspace{-3ex}

\begin{equation}\label{eq:favo}
\|A_+^{-1}\|_p\ =\  \max_{C_+\in M_+(U,V,W)}\|C_+^{-1}\|_p\ =\ \cN_{+,p}\ =\ \cN_p\ \textrm{ for all } A_+\in \PE_+(U,V,W),
\end{equation}
or \eqref{eq:favo} holds with $p$ replaced by $q$. If \eqref{eq:favo} holds we say that $p$ is {\em favourable} (for the triple $(U,V,W)$).

{\bf b)} $p$ and $q$ are both favourable iff $\cN_{+,p}=\cN_{+,q}$. If $p$ and $q$ are both favourable, then
\begin{equation}\label{eq:favopq}
\|A_+^{-1}\|_p\ =\ \cN_{+,p}\ =\ \cN_p\ =\ \cN_q\ =\ \cN_{+,q}\ =\ \|A_+^{-1}\|_q,\ \textrm{ for all } A_+\in \PE_+(U,V,W).
\end{equation}
In particular this holds for $p=q=2$.
\end{proposition}
\begin{proof}
a) Either $\|A_+^{-1}\|_p \le \cN_p$ for all $A_+\in M_+(U,V,W)$, or $\|A_+^{-1}\|_p > \cN_p$ for some $A_+\in M_+(U,V,W)$. In the first case \eqref{eq:favo} follows immediately from Proposition \ref{prop:PE+} b). In the second case, by Proposition \ref{prop:M+} b), $\|B_+^{-1}\|_q \le \cN_q$ for all $B_+\in M_+(U,V,W)$, and then \eqref{eq:favo}, with $p$ replaced by $q$, follows from Proposition \ref{prop:PE+} b).

b) is an immediate corollary of a) and Proposition \ref{prop:PE+} b).
\end{proof}

 It is unclear to us whether every $p\in[1,\infty]$ is favourable for every triple $(U,V,W)$.  Indeed, while, for every triple $(U,V,W)$, $p\in [1,\infty]$, and $A_+\in \PE_+(U,V,W)$, it follows from Propositions \ref{prop:PE+} b) and \ref{prop:M+} c) that
 \begin{equation} \label{eq:conj}
 \cN_p\ \le\ \|A_+^{-1}\|_p\ \le\ \cN_{+,p},
 \end{equation}
 it is unclear to us whether or not there are examples for which \eqref{eq:conj} holds with one or both ``$\le$'' replaced by ``$<$''. Likewise, it is unclear to us whether or not there are cases where $\|A_+^{-1}\|_p\neq \|B_+^{-1}\|_p$ with $A_+,B_+\in \PE_+(U,V,W)$. A simple case to study is that where $U$, $V$, and $W$ are singletons, in which case $M_+(U,V,W)=\PE_+(U,V,W)$ has only one element $A_+$, which is a tridiagonal Toeplitz operator. If $A_+$ is invertible, then $\cN_{+,p} = \|A_+^{-1}\|_p$.  Example 6.6 of \cite{BoeGru} shows a banded Toeplitz operator
$A_+$, for which $\|A_+^{-1}\|_p$ and $\|A_+^{-1}\|_q$ differ: if this were a {\em tridiagonal} banded Toeplitz operator then this would provide an example of a triple $(U,V,W)$ where $p$ and $q$ are not both favourable. On the other hand, the following result shows that all $p\in [1,\infty]$ are favourable for some classes of triples $(U,V,W)$.

\begin{proposition} \label{prop:favo_cases}

{\bf a)} Suppose that $p,q\in [1,\infty]$ and $p^{-1}+q^{-1}=1$. Then $p$ is favourable for $(U,V,W)$ iff $q$ is favourable for $(W,V,U)$.

{\bf b)} If $U=W$ or $0\in U\cup W$, then, for all $p\in [1,\infty]$, $p$ is favourable (for $(U,V,W)$) and \eqref{eq:favopq} holds.
\end{proposition}
\begin{proof}
a) This is clear from Proposition \ref{prop:scenarios}, since $A_+$ is in  $M_+(U,V,W)$ or $\PE_+(U,V,W)$ iff $A_+^\top $ is in $M_+(W,V,U)$ or $\PE_+(W,V,U)$, respectively, and $A_+$ is invertible iff $A_+^\top $ is invertible, in which case $\|A_+^{-1}\|_p=\|(A_+^\top)^{-1}\|_q$. \\
b) In the case that $U=W$ it follows from a) that $p$ and $q$ are both favourable.  To show that they are both favourable if $0\in W$, suppose that one of the two (say $p$) is not favourable. Then there exists $A_+\in M_+(U,V,W)$ and $x\in \ell^p(\N)$ with $\|A_+x\|_p<\cN^{-1}_p\|x\|_p$. Take any $B_+\in M_+(U,V,W)$, set $C_-=RB_+^\top R$ (as in the proof of Proposition \ref{prop:M+}), and define $B\in M(U,V,W)$ and $\tilde x\in \ell^p(\Z)$ by \eqref{eq:Bdef}, but choosing in particular $w=0\in W$, $r=1$, and $s=0$. Then $\|B\tilde x\|_p = \|A_+x\|_p < \cN^{-1}_p \|x\|_p = \cN^{-1}_p \|\tilde x\|_p$. But this implies that $\|B^{-1}\|_p > \cN_p$, a contradiction. Thus $p$ and $q$ are both favourable if $0\in W$, and it follows from part a) that they are both favourable also when $0\in U$.
\end{proof}

As one example, part b) of this proposition applies to the Feinberg-Zee random hopping matrix, $A_+\in \PE_+(U,V,W)$ or $A\in \PE(U,V,W)$ with $U=W=\{\pm 1\}$ and $V=\{0\}$ (e.g.~\cite{CWChonchaiyaLindner2011,CWDavies2011,CCL2,Hagger:NumRange,Hagger:dense,Hagger:symmetries}),  so that \eqref{eq:favopq} holds in that case (cf.~\cite[Theorem 3.6]{CCL2}).

\subsection{The relationship between finite and infinite matrices}
In this subsection we obtain more quantitative versions of the results of Section \ref{sec:FSM}. We first note the following finite version of Proposition \ref{prop:PEpq}, proved in the same way, using the observation that, for every $B\in M_n(U,V,W)$, $A=R_n B^\top  R_n \in M(U,V,W)$, where $R_n=(r_{ij})_{i,j=1,...,n}$ is the $n\times n$ matrix with $r_{ij} = \delta_{i,n+1-j}$, where $\delta_{ij}$ is the Kronecker delta.

\begin{lemma} \label{lem:PEpqfin}
For $p,q\in[1,\infty]$ with
$p^{-1}+q^{-1}=1$ and $n\in \N$, we have $\M_{n,p}=\M_{n,q}$ and $\cN_{n,p} = \cN_{n,q}$, so that $\Mfinp = \Mfinq$ and $\cNfinp = \cNfinq$, where
\begin{eqnarray*}
\M_{n,p}\ :=\ \sup_{F\in M_n(U,V,W)} \|F\|_p,\ \qquad \cN_{n,p}\ :=\ \sup_{F\in M_n(U,V,W)} \|F^{-1}\|_p,\ \\
\Mfinp := \sup_{F\in \Mfin(U,V,W)} \|F\|_p \qquad \mbox{ and } \qquad \cNfinp := \sup_{F\in \Mfin(U,V,W)} \|F^{-1}\|_p.
\end{eqnarray*}
\end{lemma}

 The following simple lemma relates $\Mfinp$ to $\M_p$, defined in Corollary \ref{cor:PE}.

\begin{lemma} \label{lem:Meq}
For $p\in [1,\infty]$,
$\Mfinp = \lim_{n\to\infty} \M_{n,p} = \M_p$.
\end{lemma}
\begin{proof} Let $A\in \PE(U,V,W)$ so that $\|A\|_p = \M_p$ by Corollary \ref{cor:PE}. For $F\in \Mfin(U,V,W)$, $\|F\|_p \le \|A\|_p$, since every $F$ is an arbitrarily small perturbation of  a finite section of $A$. On the other hand, if $A_n$ is the finite section of $A$ given by \eqref{eq:An}, then we have noted in \eqref{eq:liminf} that $\liminf_{n\to\infty} \|A_n\|_p \ge \|A\|_p$.
\end{proof}



The following is a more quantitative version of Theorem \ref{th:FSM2}:

\begin{proposition}\label{prop:F1}

{\bf a)} Properties (a)--(k) of Theorem \ref{th:FSM1} are equivalent to:
\begin{itemize}\itemsep-1mm
\item[(l)] all $F\in\Mfin(U,V,W)$ are invertible and their inverses are uniformly bounded.
\end{itemize}
If (a)--(l) are satisfied then
\begin{equation}\label{eq:F2}
\cNfinp\ =\
\max\left(\cN_{+,p}\ ,\ \cN_{+,q}\right),
\end{equation}
for every $p,q\in [1,\infty]$ with $p^{-1}+q^{-1}=1$.

{\bf b) } In the case that $p$ and $q$ in {\bf a)} are both favourable, \eqref{eq:F2} simplifies to
\begin{equation}\label{eq:favoMfinM+M}
\cNfinp\ =\ \cNfinq\ =\ \cN_{+,p}\ =\ \cN_{+,q}\ =\ \cN_p\ =\ \cN_q.
\end{equation}
\end{proposition}
\begin{proof}
{\bf a)} If $(l)$ holds then, by the equivalence of ii) and iii) in Lemma \ref{lem:FSMappl} and the definition of stability, $(e)$ holds. But this implies invertibility of all $F\in\Mfin(U,V,W)$ by Theorem \ref{th:FSM2}. The uniform boundedness of the inverses $F^{-1}$ (and hence $(l)$) will follow if we can prove ``$\le$'' in \eqref{eq:F2}.

To see that \eqref{eq:F2} holds, fix $p\in [1,\infty]$, $n\in\N$, and an $F\in M_n(U,V,W)$.
To estimate $\|F^{-1}\|_p=:f$, fix $x\in\C^n$ with $\|Fx\|_p=1$ and $\|x\|_p=f$. As in the proof of Theorem \ref{th:FSM2}, define $B$ by \eqref{eq:matB} and $B_+$ and $B_-$ as in \eqref{eq:matB+-}, and define $\tilde x$ by \eqref{eq:vecx}. Again choose $(r_k)$ as in Lemma \ref{lem:choose-rk} so that \eqref{eq:r+} or \eqref{eq:r-} holds. First assume it is \eqref{eq:r+}.

{\bf Case 1:} $p=\infty$. From \eqref{eq:r+} we get
\[
\|B_+^{-1}\|_\infty\ge\frac{\|B_+^{-1}B_+\tilde x_+\|_\infty}{\|B_+\tilde x_+\|_\infty} = \frac{\|\tilde x_+\|_\infty}{\|B_+\tilde x_+\|_\infty} = \frac{\sup_{k\in \N\cup\{0\}}|r_k|\|x\|_\infty}{\sup_{k\in \N\cup\{0\}}|r_k|\|Fx\|_\infty} = \frac{f}{1} = \|F^{-1}\|_\infty\ .
\]

{\bf Case 2:} $p<\infty$ and $\tilde x_+\in\ell^p(\N)$, i.e. $(r_k)_{k=0}^{+\infty}\in\ell^p(\N)$. Then, by \eqref{eq:r+},
\[
\|B_+^{-1}\|_p^p\ge \frac{\|B_+^{-1}B_+\tilde x_+\|_p^p}{\|B_+\tilde x_+\|_p^p} = \frac{\|\tilde x_+\|_p^p}{\|B_+\tilde x_+\|_p^p} =  \frac{\sum_{k=0}^{+\infty}|r_k|^p\|x\|_p^p}{\sum_{k=0}^{+\infty}|r_k|^p\|Fx\|_p^p} = \frac{f^p}{1^p} = \|F^{-1}\|_p^p
\]

{\bf Case 3:} $p<\infty$ and $\tilde x_+\not\in\ell^p(\N)$, i.e. $(r_k)_{k=0}^{+\infty}\not\in\ell^p(\N)$. Then $s_m:=\sum_{k=0}^{m}|r_k|^p\to\infty$ as $m\to\infty$.
Let $m\in\N$ and put $\tilde x_m:=(\tilde x(1),\tilde x(2),\cdots,\tilde x((m+1)(n+1)),0,0,\cdots)\in\ell^p(\N)$. Then
\[
\|B_+^{-1}\|_p^p\ge \frac{\|\tilde x_m\|_p^p}{\|B_+\tilde x_m\|_p^p} =  \frac{\sum_{k=0}^{m}|r_k|^p\|x\|_p^p}{\sum_{k=0}^{m}|r_k|^p\|Fx\|_p^p+|ur_mx_n|^p} = \frac{s_m f^p}{s_m+|ur_mx_n|^p}\stackrel{m\to\infty}{\longrightarrow}f^p=\|F^{-1}\|_p^p
\]
since $s_m\to\infty$ as $m\to\infty$ and $r_m$ is bounded.

So in either case we get $\|F^{-1}\|_p\le\|B_+^{-1}\|_p$ if \eqref{eq:r+} holds. The other case, \eqref{eq:r-}, is analogous and leads to $\|F^{-1}\|_p\le\|B_-^{-1}\|_p=\|C_+^{-1}\|_q$, where $C_+:=RB_-^\top R$ is the reflection of $B_-$ as discussed in Remark \ref{rem:reflect}. Since we only know that \eqref{eq:r+} or \eqref{eq:r-} applies, but not which one of them, we conclude $\|F^{-1}\|_p\le\max(\|B_+^{-1}\|_p,\|C_+^{-1}\|_q)$. Since  $F\in\Mfin(U,V,W)$ is arbitrary and $B_+,C_+\in M_+(U,V,W)$, this finishes the proof of ``$\le$'' in \eqref{eq:F2} and hence of the implication $(a)\Rightarrow(l)$. The ``$\ge$'' in \eqref{eq:F2} follows from \eqref{eq:liminf} or \eqref{eq:limsup}.

{\bf b)} follows from {\bf a)}, \eqref{eq:favo} and \eqref{eq:favopq}.
\end{proof}

Combining Proposition \ref{prop:F1} with Lemma \ref{lem:Meq} we can relate the $p$-condition numbers, $\cond_p(F)$ $:= \|F\|_p\ \|F^{-1}\|_p$ of matrices $F\in \Mfin(U,V,W)$, to the corresponding condition numbers of $A\in \PE(U,V,W)$ and $A_+\in \PE_+(U,V,W)$. For example, in the case that $p$ and $q$ are both favourable we have:

\begin{proposition} \label{prop:cond}
Suppose that $(U,V,W)$ is compatible and $\kappa(U,V,W)=0$, so that every FSM given by \eqref{eq:Anxn=b} with $l_n\to -\infty$ and $r_n\to \infty$ is applicable to $A\in \PE(U,V,W)$, and every FSM given by \eqref{eq:Anxn=b} with $l_n=1$ and $r_n\to \infty$ is applicable to $A_+\in \PE_+(U,V,W)$.  Suppose  also that $p$ and $q$ are both favourable. Then, for all these finite section methods,
\begin{eqnarray} \label{eq:conv}
\|A_n\|_p\ \to\ \|A\|_p\ =\ \|A_+\|_p\ =\ \M_p, & &\|A_n^{-1}\|_p\ \to\ \|A^{-1}\|_p\ =\ \|A_+^{-1}\|_p\ =\ \cN_p,\\
\nonumber \mbox{so that } & & \cond_p(A_n)\ \to\ \cond_p(A)\ =\ \cond_p(A_+),
\end{eqnarray}
as $n\to\infty$. Further,
\[
\max_{B \in M(U,V,W)}\ \cond_p(B)\ =\ \max_{B_+ \in M_+(U,V,W)}\ \cond_p(B_+)\
=\ \sup_{F \in \Mfin(U,V,W)}\ \cond_p(F)\ =\ \M_p\cN_p,
\]
where the two maxima are attained, respectively, by all $A\in \PE(U,V,W)$, and by all $A_+\in \PE(U,V,W)$.
\end{proposition}
\begin{proof}
The first sentence is part of Theorem \ref{th:FSM1}. And if $p$ and $q$ are both favourable, then the equalities in \eqref{eq:conv} follow from Corollary \ref{cor:PE}, Proposition \ref{prop:PE+} a), and \eqref{eq:favopq}.
The stated limits are a consequence of
\[
\M_p=\|A\|_p\le\liminf_{n\to\infty}\|A_n\|_p\le\limsup_{n\to\infty}\|A_n\|_p\le\Mfinp=\M_p
\]
and
\[
\cN_p=\cN_{+,p}=\|A_+^{-1}\|_p\le\liminf_{n\to\infty}\|A_n^{-1}\|_p\le\limsup_{n\to\infty}\|A_n^{-1}\|_p\le\cNfinp=\cN_p,
\]
by \eqref{eq:liminf} and that, by Lemma \ref{lem:Meq} and Proposition \ref{prop:F1} b), $\Mfinp = \M_p$ and $\cNfinp = \cN_p$. In the last displayed equation the first equality, and that these maxima are attained as stated and have the value $\M_p\cN_p$, follows from Corollary \ref{cor:PE} and Propositions \ref{prop:PE+} a) and \ref{prop:scenarios} b). That $\sup_{F \in \Mfin(U,V,W)} \cond_p(F)\le  \M_p\cN_p$ is clear from $\Mfinp\cNfinp = \M_p\cN_p$; that in fact equality holds is clear from \eqref{eq:conv}.
\end{proof}

\subsection{Pseudospectra}
We can rephrase our results on the norms of inverses, $\|J^{-1}\|$, of Jacobi matrices $J$ over $(U,V,W)$ in terms of resolvent norms $\|(J-\lambda I)^{-1}\|$ and pseudospectra, noting that $J-\lambda I$ is a Jacobi matrix over $(U,V-\lambda,W)$. In particular, $J$ and $J-\lambda I$ are both pseudoergodic at the same time. In the language of pseudospectra, Corollary \ref{cor:PE} and Proposition \ref{prop:PEpq} can be rewritten as follows:
\begin{corollary} {\bf -- bi-infinite matrices. }\label{cor:speps.PE}
For all $A\in\PE(U,V,W)$, $\eps>0$ and $p\in[1,\infty]$, it holds that
\[
\speps^p A\ =\ \S_\eps^p\ :=\ \bigcup_{B\in M(U,V,W)}\speps^p B
\qquad\textrm{and}\qquad \S_\eps^p\ =\ \S_\eps^q,
\]
where $p^{-1}+q^{-1}=1$.
\end{corollary}

Summarizing Corollary \ref{cor:specB+}, Theorem \ref{thm:gap} and the results in Section \ref{sec:M+vsM}, and recalling the notations $\Eins$ and $\Eall$ from \eqref{eqs:E}, we obtain:
\begin{proposition} {\bf -- semi- vs. bi-infinite matrices. }\label{prop:speps.M+PE+}

{\bf a)} For every $A_+\in\PE_+(U,V,W)$, $\eps>0$ and all $p\in[1,\infty]$, it holds that
\begin{equation} \label{eq:unionPS+}
\S_\eps^p\ \cup \Sigma_+ \subset\ \speps^p A_+\ \subset\ \Sigma_{+,\eps}^p := \bigcup_{C_{+}\in M_{+}(U,V,W)}\speps^{p}C_{+}.
\end{equation}

{\bf b)} For all $A_+,B_+\in M_+(U,V,W)$, $\eps>0$ and $p,q\in[1,\infty]$ with $p^{-1}+q^{-1}=1$, it holds that
\begin{equation}\label{eq:spepsM+}
\speps^p A_+\ \cap\ \speps^q B_+\ \subset\ \Sigma_{+,\eps}^p\ \cap\ \Sigma_{+,\eps}^q\ =\ \S_\eps^p\ \cup\ G,
\end{equation}
 for each $G\in \{E_{\pm 1}, \Eins,\Eall, \S_{+}\}$.  Equality holds in \eqref{eq:spepsM+} if $A_+,B_+\in\PE_+(U,V,W)$.  If $w_*\le u^*$ and $u_*\le w^*$, then $E_{\pm 1}=\varnothing$, so that \eqref{eq:spepsM+} holds with $G=\varnothing$.

{\bf c) } If $\eps>0$ and $p\in[1,\infty]$ is favourable, in particular if $p=2$, then, for every $A_{+}\in\PE_{+}(U,V,W)$,
\begin{equation}\label{eq:favoc}
\speps^{p}A_{+}\ =\ \Sigma_{+,\eps}^p\ =\ \S_\eps^{p}\ \cup\ G,
\end{equation}
 for each $G\in \{E_{\pm 1}, \Eins,\Eall, \S_{+}\}$. In particular \eqref{eq:favoc} holds with $G=\varnothing$ if $w_*\le u^*$ and $u_*\le w^*$.

{\bf d)} If $p,q\in[1,\infty]$ with $p^{-1}+q^{-1}=1$ are both favourable, in particular if $0\in U\cup W$ or $U=W$, then \eqref{eq:favoc} holds for both $p$ and $q$, and $\speps^{p}A_{+}=\speps^{q}A_{+}$.
\end{proposition}
Regarding the different possibilities for $G$ in the above proposition, recall from Theorem \ref{thm:gap} that each indicated choice closes the gap between $\S$ and $\S_{+}$ in the sense that $\S\cup G=\S_{+}$, including the choice $G=\varnothing$ if $w_*\le u^*$ and $u_*\le w^*$.
\begin{proof} First, note that $\S\subset\S_\eps^p$ since $\S=\spec A\subset\speps^p A=\S_\eps^p$ for every $A\in\PE(U,V,W)$.

{\bf a) } Let $A_+\in\PE_+(U,V,W)$. If $\lambda\in\S_+=\spec A_+$ then $\lambda\in\speps^{p}A_{+}$. If $\lambda\in\S_{\eps}^{p}\setminus\S\supset \S_{\eps}^{p}\setminus\S_+$ then, for every $A\in\PE(U,V,W)$, $\|(A_{+}-\lambda I_{+})^{-1}\|_{p}\ge\|(A-\lambda I)^{{-1}}\|_{p}>\eps^{-1}$, by Proposition \ref{prop:PE+} b), so that $\lambda\in\speps^{p}A_{+}$.

{\bf b) } We start with the equality in \eqref{eq:spepsM+}. By $\S\subset\S_\eps^p$ and Theorem \ref{thm:gap}, $\S_\eps^p\cup G = \S_\eps^p\cup\S\cup G = \S_\eps^p\cup\S_+$, for each $G\in \{E_{\pm 1}, \Eins,\Eall, \S_{+}\}$.

To show that $\speps^p A_+ \cap \speps^q B_+ \subset \S_\eps^p \cup \S_+$, 
let $A_+,B_+\in M_+(U,V,W)$ and $\lambda\in\speps^p A_+ \cap \speps^q B_+$, so that $\|(A_+-\lambda I_+)^{-1}\|_p>\eps^{-1}$ and $\|(B_+-\lambda I_+)^{-1}\|_q>\eps^{-1}$. If one of these two numbers is infinite, i.e.~$\lambda\in\spec A_+$ or $\lambda\in\spec B_+$, then 
$\lambda\in\S_+$. 
 If both $\|(A_+-\lambda I_+)^{-1}\|_p$ and $\|(B_+-\lambda I_+)^{-1}\|_q$ are finite and greater than $\eps^{-1}$, consider $A\in\PE(U,V,W)$. If $A-\lambda I$ is not invertible then $\lambda$ is in $\S\subset \S_+$. If $A-\lambda I$ is invertible then, by \eqref{eq:M+min}, $\|(A-\lambda I)^{-1}\|\ge\min(\|(A_+-\lambda I_+)^{-1}\|_p,\|(B_+-\lambda I_+)^{-1}\|_q)>\eps^{-1}$, so that $\lambda$ is in $\S_\eps^p$.

Now suppose that $A_+,B_+\in\PE_+(U,V,W)$ 
and that $\lambda\in\S_\eps^p\cup\S_+$. If $\lambda\in\S_+$, then $\lambda\in\spec A_+=\spec B_+$ by Corollary \ref{cor:specB+}, so that $\lambda\in\speps^p A_+ \cap \speps^q B_+$. If $\lambda\in\S_\eps^p$ then, by Proposition \ref{prop:M+} c) and \eqref{eq:||PE||}, $\min(\|(A_+-\lambda I_+)^{-1}\|_p,\|(B_+-\lambda I_+)^{-1}\|_q)=\|(A-\lambda I)^{-1}\|>\eps^{-1}$ for every $A\in\PE(U,V,W)$, so, again, $\lambda\in\speps^p A_+ \cap \speps^q B_+$. In both cases it follows that $\lambda\in \S_{+,\eps}^p\cap \S_{+,\eps}^q$.

{\bf c) } This follows immediately from {\bf b)} and \eqref{eq:favo}.

{\bf d) } This follows from {\bf c)}, \eqref{eq:favopq},  and Proposition \ref{prop:favo_cases} b).
\end{proof}

Here is the pseudospectral formulation of Proposition \ref{prop:F1}:

\begin{corollary}{\bf -- finite, semi- and bi-infinite matrices.}\label{cor:spepsMfin}

{\bf a) } For all $\eps>0$ and $p,q\in[1,\infty]$ with $p^{-1}+q^{-1}=1$, it holds that
\[
\Sfin^p\ :=\ \bigcup_{F\in\Mfin(U,V,W)}\speps^p F\ =\ \S_{+,\eps}^p\ \cup\ \S_{+,\eps}^q.
\]

{\bf b) } If $p$ and $q$ are both favourable, in particular if $p=q=2$,  this simplifies to
\begin{equation} \label{eq:final}
\Sfin^p\  =\
\speps^p A_+\ =\ \S_{+,\eps}^p\ =\ \S_\eps^p\ \cup\ G\ =\ \S_\eps^q\ \cup\ G\ =\ \S_{+,\eps}^q\ =\ \speps^q B_+,
\end{equation}
which holds for all $A_+,B_+\in\PE_+(U,V,W)$,  and for each $G\in \{E_{\pm 1}, \Eins,\Eall, \S_{+}\}$, with $E_{\pm 1}=\varnothing$ if $w_*\le u^*$ and $u_*\le w^*$.
\end{corollary}
\begin{proof}
{\bf a)} This follows immediately from Proposition \ref{prop:F1}. In particular,
\begin{align*}
\sup_{F\in\Mfin(U,V,W)}&\|(F-\lambda I)^{-1}\|_p\\
&=\ \max\left(\sup_{A_+\in M_+(U,V,W)}\|(A_+-\lambda I_{+})^{-1}\|_p\ ,\ \sup_{A_+\in M_+(U,V,W)}\|(A_+-\lambda I_{+})^{-1}\|_q\right)
\end{align*}
holds for all $\lambda\in\C$, where both sides are infinite at the same time.

{\bf b)} follows from {\bf a)} and Proposition \ref{prop:speps.M+PE+} d).
\end{proof}
\medskip

Finally we study asymptotics of the pseudospectra of the finite sections of $A_+\in$ $\PE_+(U,V,W)$, generalising results from \cite{CCL2} for the Feinberg-Zee random hopping matrix, and results of \cite{BoeGruSilb97} which apply when $U$, $V$, and $W$ are singletons. For this purpose we briefly recall notions of set convergence. Let $\C^B, \C^C$ denote the sets of bounded and compact, respectively, non-empty subsets of $\C$.  For $a\in \C$ and non-empty $B\subset \C$, let $\dist(a,B) :=\inf_{b\in B}|a-b|$. For $A,B\in \C^B$ let
\begin{equation} \label{eq:Haus}
d_H(A,B)\ :=\ \max\left(\sup_{a\in A}\dist(a,B),\sup_{b\in B}\dist(b,A)\right).
\end{equation}
It is well known (e.g., \cite{Hausdorff,HaRoSi2}) that $d_H(\cdot,\cdot)$ is a metric on $\C^C$, the so-called {\em Hausdorff metric}. For $A, B\in \C^B$ it is clear that $d_H(A,B)=d_H(\overline{A},\overline{B})$, so that $d_H(A,B)=0$ iff $\overline{A}=\overline{B}$, and $d_H(\cdot,\cdot)$ is a pseudometric on $\C^B$. For a sequence $(S_n)\subset\C^B$ and $S\in \C^B$ we write $S_n\toH S$ if $d_H(S_n,S)\to 0$. This limit is in general not unique: if $S_n\toH S$, then $S_n\toH T$ iff $\overline{S}=\overline{T}$.

A second related standard notion of set convergence is the following \cite[Section 28]{Hausdorff}, \cite[Definition 3.1]{HaRoSi2}: for sets $S_n, S\subset\C$, we write $S_n\to S$ if $\liminf S_n=\limsup S_n=S$, where $\liminf S_n$ is the set of limits of sequences $(z_n)\subset \C$ such that $z_n\in S_n$ for each $n$, while $\limsup S_n$ is the set of partial limits of such sequences. Both $\liminf S_n$ and $\limsup S_n$ are closed sets \cite[Proposition 3.2]{HaRoSi2}, and it is clear (see \cite[Proposition 3.5]{HaRoSi2}) that, for every set sequence $(S_n)$, $\liminf S_n=\liminf \overline{S_n}$ and $\limsup S_n=\limsup \overline{S_n}$. These two notions of convergence are very close. Precisely,  for $(S_n)\subset \C^B$, $\liminf S_n=\limsup S_n$ iff $\liminf \overline{S_n} =\liminf S_n=\limsup S_n=\limsup \overline{S_n}$ and (see \cite[p.~171]{Hausdorff} or \cite[Proposition 3.6]{HaRoSi2}), this holds iff $\overline{S_n}\toH S$ for some $S\in \C^C$, in which case $\overline{S_n}\to S$. Further, $\overline{S_n}\toH S$ iff $S_n\toH S$, and $\overline{S_n}\to S$ iff $S_n\to S$. Thus, for $(S_n)\subset \C^B$ and $S\in \C^B$, $S_n\toH S$ iff $S_n\to \overline{S}$.

We introduce the following stronger notion of set convergence (cf., \cite[p.~21]{Hausdorff}). Given
$S\subset\C$ and a sequence of sets $S_n\subset \C$,  we will write $S_n\nearrow S$ if $S_n\subset S$ for each $n$, and if every $z\in S$ is also in $S_n$ for all sufficiently large $n$. In symbols $S_n\nearrow S$ means that
$S_n \subset S$ for each $n$ and that $S = \cup_{m\in\N}\cap_{n\ge m}S_n$.
The following lemma is immediate from this definition and the observations above:
\begin{lemma} \label{lem:setconv} If $S_n\subset \C$ is a set sequence, $S\subset \C$, and $S_n\nearrow S$, then $S_n\to \overline{S}$. If, additionally, each $S_n$ is non-empty and $S$ is bounded, then also $S_n\toH S$.
\end{lemma}

In this last proposition we use again the notation \eqref{eq:Spec}. We remark that a weaker version of this proposition, namely that, under the same assumptions on $p$ and $q$, $\limsup \mathrm{spec}_\eps^p A_n = \mathrm{Spec}_\eps^p A_+$, can be deduced from general results on the finite section method for band-dominated operators \cite{SeidelSilb13}, results which, as already noted, also lead to \eqref{eq:limsup}.
\begin{proposition} \label{prop:speps_conv}
Let $A_+\in\PE_+(U,V,W)$ and let $(A_n)$ denote its finite sections \eqref{eq:An} with $l_n= 1$ and $r_n\to +\infty$. If $p,q\in[1,\infty]$ with $p^{-1}+q^{-1}=1$ are both favourable, in particular if $p=q=2$ or $0\in U\cup W$ or $U=W$, then, for all $\eps>0$,
$$
\speps^p A_n\ \nearrow\ \speps^p A_+\ =\ \S_{+,\eps}^p,\qquad \mbox{ as }\ n\to\infty,
$$
so that also $\speps^p A_n \to \overline{\speps^p A_+}= \mathrm{Spec}_\eps^p A_+$ and $\speps^p A_n \toH \speps^p A_+$.
\end{proposition}
\begin{proof}
By \eqref{eq:final}, $\speps^p A_n\subset\speps^p A_+=\S_{+,\eps}^p$ for all $n\in\N$. Further, for every $0<\eta< \eps$, it follows from \cite[Theorem 4.4]{CCL2} that $\spec_\eta^p A_+\subset \speps^p A_n$ for all sufficiently large $n$,  i.e.~that, for some $m\in\N$, $\spec_\eta^p A_+\subset\cap_{n\ge m} \speps^p A_n$. Since $\cup_{0<\eta<\eps}\spec_\eta^p A_+ = \spec_\eps^p A_+$, it follows that $\spec_\eps^p A_+=\cup_{m\in\N}\cap_{n\ge m} \speps^p A_n$, and hence that $\speps^p A_n\nearrow  \spec_\eps^p A_+$ as $n\to\infty$. The remaining results follow from Lemma \ref{lem:setconv}.
\end{proof}

\begin{remark}{\bf \; -- the Hilbert space case $p=2$ and the numerical range.}\label{rem:HSC}
We emphasise that our results are somewhat simpler in the important Hilbert space case $p=2$. In particular, for $A_+\in \PE_+(U,V,W)$ and $\eps>0$, where $(A_n)$ denotes the finite sections \eqref{eq:An} of $A_+$ with $l_n= 1$ and $r_n\to +\infty$, we have from \eqref{eq:final} and Proposition \ref{prop:speps_conv} that
$$
\speps^2 A_n\ \nearrow\ \speps^2 A_+\ =\ \Sfin^2\  =\
 \S_{+,\eps}^2\ =\ \S_\eps^2\ \cup\ G,
$$
for each $G\in \{E_{\pm 1}, \Eins,\Eall, \S_{+}\}$, with $E_{\pm 1}=\varnothing$ if $w_*\le u^*$ and $u_*\le w^*$.   Further, in this Hilbert space case, as discussed above \eqref{eq:hagger}, Hagger \cite{Hagger:NumRange} has shown that the closure of the numerical range of every operator in $\PE(U,V,W)\cup \PE_+(U,V,W)$ is given explicitly by $\conv(E)=\conv(E_\cup)$. Since, for a bounded operator on a Hilbert space, the $\eps$-neighbourhood of the numerical range contains the $\eps$-pseudospectrum (e.g.~\cite[Theorem 17.5]{TrefContEmb}), and recalling \eqref{eq:specenc}, we can extend the lower and upper bounds \eqref{eq:hagger} for the spectrum to give bounds on the 2-norm pseudospectra, that, for $\eps>0$,
$$
E_\cup + \eps \D\ \subset\ \S_{+,\eps}^2\ \subset\ \conv(E_\cup) + \eps \D\ \quad \mbox{ and } \quad E + \eps \D\ \subset\ \S_{\eps}^2\ \subset\ \conv(E) + \eps \D.
$$
Example \ref{ex:possi}, with $U=\{1\}$, $V=\{0\}$, $W=\{0,2\}$ and $E_\cup = \conv E_\cup$, is a case where these bounds on $\S_{+,\eps}^2$ are sharp: in that case $\S_{\eps}^2 =\S_{+,\eps}^2= \conv(E_\cup) + \eps \D = \conv E(1,2) + \eps \D$, with $E(1,2)$ the ellipse \eqref{eq:elli12}.
\end{remark}


\bigskip
{\bf Acknowledgements.}
The authors are grateful for funding as the work was developed, and opportunities for interaction and feedback, through the EPSRC MOPNET Network EP/G01387X/1. The second author would also like to thank Raffael Hagger and Christian Seifert (TU Hamburg) for helpful discussions.


\noindent {\bf Authors' addresses:}\\
\\
Simon~N.~Chandler-Wilde \hfill {\tt S.N.Chandler-Wilde@reading.ac.uk} \\
Department of Mathematics and Statistics\\
University of Reading\\
Reading, RG6 6AX\\
UK\\
\\
Marko Lindner\hfill {\tt lindner@tuhh.de}\\
Institut Mathematik\\
TU Hamburg (TUHH)\\
D-21073 Hamburg\\
GERMANY

\end{document}